\documentclass[a4paper,11pt]{article}

\usepackage[a4paper,left=1.9cm,right=1.9cm,top=2.0cm,bottom=2cm]{geometry}
\usepackage{enumitem}
\usepackage[dvips]{graphicx}
\usepackage{epsfig}
\usepackage{graphicx}
\usepackage{calc}
\usepackage{cite}
\usepackage{bm}
\usepackage{amsmath,amssymb,latexsym}
\usepackage{amsmath,amsthm}
\usepackage{cases,setspace}
\usepackage{titlesec}
\usepackage{tikz}
\usepackage{enumitem}

\usepackage{lineno}
\usepackage{amsmath}
\usepackage{epsfig}
\usepackage{rotating}
\usepackage[hang,stable]{footmisc}
\usepackage{color}
\usepackage{dsfont}
\usepackage{multirow}
\usepackage{ifpdf}
\usepackage{amsthm}
\usepackage{cases}
\usepackage{booktabs}
\usepackage{subeqnarray}
\usepackage{amssymb}
\usepackage{epsfig}

\usepackage{subfigure}
\def\be{\begin{equation}}
\def\ee{\end{equation}}
\def\bea{\begin{eqnarray}}
\def\eea{\end{eqnarray}}
\def\bea*{\begin{eqnarray*}}
\def\eea*{\end{eqnarray*}}
\def\bt{\begin{theorem}}
\def\et{\end{theorem}}
\def\bl{\begin{lemma}}
\def\el{\end{lemma}}
\def\br{\begin{remark}}
\def\er{\end{remark}}
\def\bc{\begin{corollary}}
\def\ec{\end{corollary}}
\def\bd{\begin{definition}}
\def\ed{\end{definition}}
\def\bd{\begin{algorithm}}
\def\ed{\end{algorithm}}

\numberwithin{equation}{section}
\numberwithin{figure}{section}

\theoremstyle{plain}
\newtheorem{theorem}{Theorem}[section]
\newtheorem{corollary}{Corollary}[section]
\newtheorem{lemma}{Lemma}[section]
\newtheorem{remark}{Remark}[section]

\newtheorem*{problem1}{Problem P$_1$}
\newtheorem*{problem2}{Problem P$_2$}

\newtheorem{algorithm}{Algorithm}[section]

\titlespacing*{\subsubsection}{0pt}{*1.0}{*1.0}
\titleformat*{\subsubsection}{\itshape}

\begin{document}
\title{A Gradient-based Kernel Optimization Approach for Parabolic Distributed Parameter Control Systems\footnote{This work was supported by the National Natural Science Foundation of China (F030119-61104048, 61320106009) and Fundamental Research Funds for the Central Universities (2014FZA5010).}}
\author{Zhigang Ren$^{1}$, Chao Xu$^{1,}$\footnote{Correspondence to: Chao Xu, Email: cxu@zju.edu.cn}, Qun Lin$^2$, and Ryan Loxton$^{2}$}
\date{}
\maketitle
\begin{center}
{\small
$^1$State Key Laboratory of Industrial Control Technology and Institute of Cyber-Systems \& Control, Zhejiang University, Hangzhou, Zhejiang, China\\
$^2$Department of Mathematics \& Statistics, Curtin University, Perth, Western Australia, Australia}
\end{center}

\begin{abstract}
This paper proposes a new gradient-based optimization approach for designing optimal feedback kernels for parabolic distributed
parameter systems with boundary control. Unlike traditional kernel optimization methods for parabolic systems, our new method does not require solving non-standard Riccati-type or Klein-Gorden-type partial differential equations (PDEs). Instead, the feedback kernel is parameterized as a second-order polynomial whose coefficients are decision variables to be tuned via gradient-based dynamic optimization, where the gradient of the system cost functional (which penalizes both kernel and output magnitude) is computed by solving a so-called ``costate'' PDE in standard form. Special constraints are imposed on the kernel coefficients to ensure that, under mild conditions, the  optimized kernel yields closed-loop stability.
Numerical simulations demonstrate the effectiveness of the proposed approach.
\\
\\[-0.5\baselineskip]
\end{abstract}

\section{Introduction} \label{S:intro}

Spatial-temporal evolutionary processes (STEPs) are usually modeled by partial differential equations (PDEs), which are commonly referred to  as distributed parameter systems (DPS) in the field of automatic control. The parabolic system is an important type of DPS that describes a wide range of natural phenomena, including diffusion, heat transfer,  and fusion plasma transport. Over the past few decades, control theory for the parabolic DPS has developed into a mature research topic at the interface of engineering and applied mathematics~\cite{curtain1995introduction, bensoussan2007, liu2010elementary}.

There are two main control synthesis approaches for the parabolic DPS: the \textit{design-then-discretize} framework and the \textit{discretize-then-design} framework.
In the design-then-discretize framework, analytical techniques are first applied to derive equations for the optimal controller---e.g., the Riccati equation for optimal control or the backstepping kernel equation for boundary stabilization---and  then these equations are  solved using numerical discretization techniques. In the discretize-then-design framework, the sequence is reversed: the infinite dimensional PDE system is first discretized to obtain a finite dimensional system, and then various controller synthesis and numerical optimization techniques are applied to solve the corresponding discretized problem. Popular approaches for obtaining the finite dimensional approximation in the discretize-then-design framework include  mesh-associated discretization techniques and model order reduction (MOR) techniques. Examples of mesh-associated discretization techniques include the finite difference method~\cite{strikwerda2007}, the finite element method~\cite{strikwerda2007}, the finite volume method~\cite{strikwerda2007}, and the spectral method~\cite{ferziger1996}. MOR methods include the proper orthogonal decomposition method~\cite{holmes1998turbulence} and the balanced truncation method~\cite{moore1981principal}, both of which exploit system input-output properties~\cite{antoulas2005app}. Since MOR techniques can generate low-order models without compromising solution accuracy, they are  popular for dealing with complex STEPs, which arise frequently in applications such as plasma physics, fluid flow, and heat and mass transfer (e.g., see~\cite{xu2011seq} and the references cited therein).

The linear quadratic (LQ) control framework, a widely-used technique in controller synthesis, is well-defined in infinite dimensional function spaces to deal with the parabolic DPS (e.g.,~\cite{bensoussan2007,curtain1995introduction}). However, the LQ control framework requires solving Riccati-type differential equations, which are nonlinear parabolic PDEs of dimension one greater than the original parabolic PDE system. For example, to generate an optimal feedback controller for a scalar heat equation, a Riccati PDE defined over a rectangular domain must be solved~\cite{moura2013optimal}. Hence, the LQ approach does not actually solve the controller synthesis problem directly, but instead converts it into another problem (i.e., solve a Riccati-type PDE) that is still extremely difficult to solve from a computational point of view.

One of the major advances in PDE control in recent years has been the so-called infinite dimensional Voltera integral feedback, or the backstepping method~(e.g., \cite{kr4, liu}). Instead of Riccati-type PDEs, the backstepping method requires solving the so-called kernel equations---linear Klein-Gorden-type PDEs for which the successive approach can be used to obtain explicit solutions. This method was originally developed for the stabilization of one dimensional parabolic DPS and then extended to fluid flows~\cite{vazquez2007}, magnetohydrodynamic flows~\cite{vazquezM, xu2008stab}, and elastic vibration~\cite{Krstic2008}. In addition, the backstepping method can also be applied to achieve full state feedback stabilization and state estimation of PDE-ODE cascade systems~\cite{antonio2010control}.

In this paper, we propose a new framework for control synthesis for the parabolic DPS. This new framework does not require solving Riccati-type or Klein-Gorden-type PDEs. Instead, it requires solving a so-called ``costate'' PDE, which is much easier to solve from a computational viewpoint. In fact, many numerical software packages, such as Comsol Multiphysics and MATLAB PDE ToolBox, can be used to generate numerical solutions for the costate PDE. The Riccati PDEs, on the other hand, are usually not in standard form and thus cannot be solved using off-the-shelf software packages. The optimization approach proposed in this paper is motivated by the well-known PID tuning problem, in which the coefficients in a PID controller need to be selected judiciously to optimize system performance. Relevant literature includes reference \cite{kill2006}, where extremum seeking algorithms are used to tune the PID parameters; reference \cite{li2011optimal}, where the PID tuning problem is reformulated into a nonlinear optimization problem, and subsequently solved using numerical optimization techniques; and reference~\cite{xu2007optimal}, where the iterative learning tuning method is used to update the PID parameters whenever a control task is repeated.
The current paper can be viewed as an extension of these optimization-based feedback design ideas to infinite dimensional systems.

The remainder of this paper is organized as follows. In Section \ref{sec:statement}, we formulate two parameter optimization problems for a class of unstable linear parabolic diffusion-reaction PDEs with control actuation at the boundary: the first problem involves optimizing a set of parameters that govern  the feedback kernel; the second problem is a modification of the first problem with additional constraints to ensure closed-loop stability. In Section~\ref{sec:VariationalAnalysis}, we derive the gradients of the cost and constraint functions for the optimization problems in Section \ref{sec:statement}. Then, in Section \ref{sec:NumericalAlgorithms}, we present a numerical algorithm, which is based on the results obtained in Section \ref{sec:VariationalAnalysis}, for determining the optimal feedback kernel. Section~\ref{sec:Numerical Simulations}  presents the numerical simulation results. Finally, Section \ref{sec:conclusion} concludes the paper by proposing some further research topics.

\section{Problem Formulation}\label{sec:statement}
\subsection{Feedback Kernel Optimization}
We consider the following parabolic PDE system:
\begin{subequations}\label{pde0}
\begin{numcases}{}
y_t(x,t) = y_{xx}(x,t) + cy(x,t),\\
y(0,t) = 0,\\
y(1,t) = u(t),\\
y(x,0) = y_0(x),
\end{numcases}
\end{subequations}
where $c>0$ is a given constant and $u(t)$ is a boundary control. It is well known that the uncontrolled version of system (\ref{pde0}) is unstable when the constant $c$ is sufficiently large~\cite{kr4}. Thus, it is necessary to design an appropriate feedback control law for $u(t)$ to stabilize the system.
According to the LQ control~\cite{moura2013optimal} and backstepping synthesis approaches~\cite{kr4}, the optimal feedback control law takes the following form:
\begin{equation}\label{equfeedback}
u(t)=\int_0^1\mathcal K(1,\xi)y(\xi,t)d\xi,
\end{equation}
where the feedback kernel $\mathcal K(1,\xi)$ is obtained by solving either a Riccati-type or a Klein-Gorden-type PDE. By introducing the new notation $k(\xi)=\mathcal K(1,\xi)$, we can write the feedback control policy (\ref{equfeedback}) in the following form:
\begin{equation}\label{ukenerl}
u(t)=\int_0^1 k(\xi)y(\xi,t)d\xi.
\end{equation}
The corresponding closed-loop system is
\begin{subequations}\label{closed-pde}
\begin{numcases}{}
y_t(x,t) = y_{xx}(x,t) + cy(x,t), \label{closed-pdef}\\
y(x,0) = y_0(x), \label{closed-pdeini}\\
y(0,t) = 0, \label{closed-pdeifl}\\
y(1,t) = \int_0^1 k(\xi)y(\xi,t)d\xi. \label{cpdefr11}
\end{numcases}
\end{subequations}
In reference~\cite{kr4}, the backstepping method is used to express the optimal feedback kernel in terms of the first-order modified Bessel function. More specifically,
\be
\mathcal K(1,\xi)=-c\xi \frac{I_1(\sqrt{c(1-\xi^2)})}{\sqrt{c(1-\xi^2)}}, \label{bsKernel}
\ee
where $I_1$ is the first-order modified Bessel function given by
\be
I_1(\omega)=\sum_{n=0}^{\infty}\frac{\omega^{2n+1}}{2^{2n+1}n!(n+1)!}.\nonumber
\ee
The feedback kernel (\ref{bsKernel}) is plotted  in Figure \ref{fig:kernelbackst} for different values of $c$. Note that its shape is similar to a quadratic function. Note also that $\mathcal K(1,\xi)=0$ when $\xi=0$. Accordingly, motivated by the quadratic behavior exhibited in Figure \ref{fig:kernelbackst}, we express $k(\xi)$ in the following parameterized form:
\begin{equation}\label{kequation}
k(\xi;\Theta) = \theta_1 \xi + \theta_2 \xi^2,
\end{equation}
where $\Theta=(\theta_1,\theta_2)^\top$ is a parameter vector to be optimized.

Moreover, we assume that the parameters must satisfy the following bound constraints:
\begin{equation}\label{kernelbound}
a_1\leq\theta_1\leq b_1, \quad a_2\leq\theta_2\leq b_2,
\end{equation}
where $a_1$, $a_2$, $b_1$ and $b_2$ are given bounds.
\begin{figure}
\begin{center}
\includegraphics[width=4.0in]{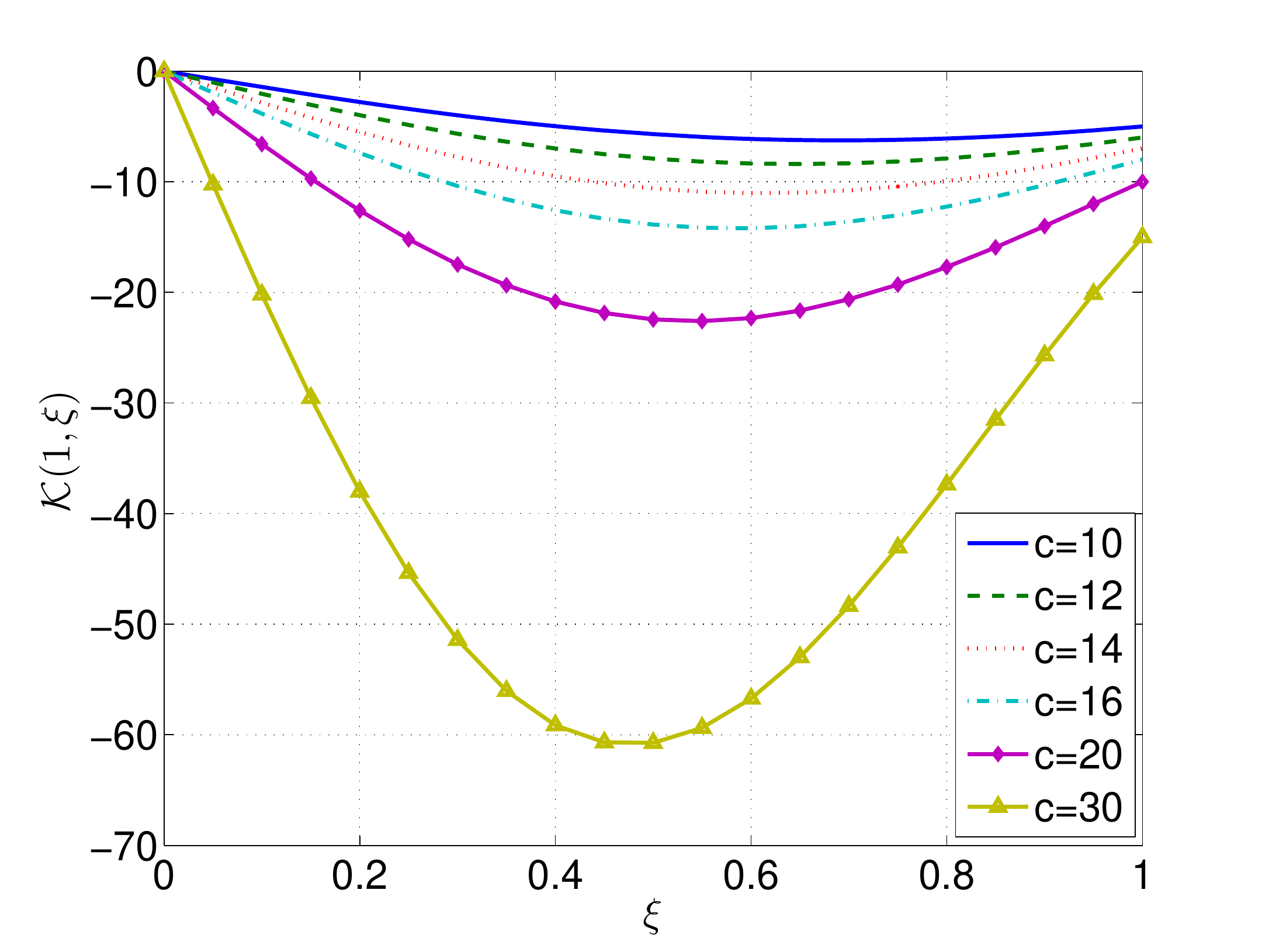}
\caption{The feedback kernel (\ref{bsKernel}) for various values of $c$.}\label{fig:kernelbackst}
\end{center}
\end{figure}

Let 
$y(x,t;\Theta)$ denote the solution of the closed-loop system (\ref{closed-pde}) with the parameterized kernel (\ref{kequation}). The results in \cite{triggiani1980well} ensure that such a solution exists and is unique.
Our goal is to stabilize the closed-loop system with minimal energy input. Accordingly, we consider the following cost functional:
\begin{equation}\label{optcostfun}
\begin{aligned}
g_0(\Theta)=\frac{1}{2}\int_0^T\int_0^1 y^2(x,t;\Theta)dxdt+\frac{1}{2}\int_0^1k^2(x;\Theta)dx.\\
\end{aligned}
\end{equation}
This cost functional consists of two terms: the first term penalizes output deviation from zero (stabilization); the second term penalizes kernel magnitude (energy minimization).
We now state our kernel optimization problem formally as follows.
\begin{problem1}
Given the PDE system (\ref{closed-pde}) with the parameterized kernel (\ref{kequation}), find an optimal parameter vector $\Theta=(\theta_1, \theta_2)^\top$ such that the cost functional (\ref{optcostfun}) is minimized subject to the bound constraints (\ref{kernelbound}).
\end{problem1}

\subsection{Closed-Loop Stability}\label{sec:stability verification}
Since (\ref{optcostfun}) is a finite-time  cost functional, there is no guarantee that the optimized kernel (\ref{kequation}) generated by the solution of Problem P$_1$ stabilizes the closed-loop system (\ref{closed-pde}) as $t\rightarrow\infty$. ~~Nevertheless, we now show that, by analyzing the solution structure of (\ref{closed-pde}), ~additional constraints can be added ~to ~Problem ~P$_1$ to ensure closed-loop stability. 

Using the separation of variables approach, we decompose $y(x,t)$ as follows:
\begin{equation} \label{definitionY}
y(x,t) = \mathcal X(x)\mathcal{T}(t).
\end{equation}
Substituting (\ref{definitionY}) into (\ref{closed-pdef}), we obtain
\begin{equation}\label{changed-pde1}
\begin{aligned}
\mathcal X(x)\dot {\mathcal{T}}(t) = \mathcal X''(x)\mathcal{T}(t)+c \mathcal X(x)\mathcal{T}(t),\\
\end{aligned}
\end{equation}
where
\begin{equation*}
\begin{aligned}
\dot {\mathcal{T}}(t)&=\frac{d\mathcal{T}(t)}{dt},\\
\mathcal X''(x)&=\frac{d^2 \mathcal X(x)}{dx^2}.
\end{aligned}
\end{equation*}
Furthermore, from the boundary conditions (\ref{closed-pdeifl}) and (\ref{cpdefr11}),
\begin{equation}\label{changed-pdeinitial}
\begin{aligned}
\mathcal X(0)\mathcal{T}(t) = 0,\quad \mathcal X(1)\mathcal{T}(t) = \int_0^1 k(\xi;\Theta)\mathcal X(\xi)\mathcal{T}(t)d\xi.
\end{aligned}
\end{equation}
Thus, we immediately  obtain
\begin{equation} \label{definition-inY}
\mathcal X(0)=0,
\end{equation}
\begin{equation} \label{definition-inY1}
\mathcal X(1)=\int_0^1k(\xi;\Theta)\mathcal X(\xi)d\xi.
\end{equation}
Rearranging (\ref{changed-pde1}) gives
\begin{equation*}
\frac{\mathcal X''(x)+c\mathcal X(x)}{\mathcal X(x)}=\frac{\dot {\mathcal{T}}(t)}{\mathcal{T}(t)}.
\end{equation*}
This equation must hold for all $x$ and $t$. Hence, there exists a constant $\sigma$ (an eigenvalue) such that
\begin{equation} \label{definitionY1}
\frac{\mathcal X''(x)+c\mathcal X(x)}{\mathcal X(x)}=\frac{\dot {\mathcal{T}}(t)}{\mathcal{T}(t)}=\sigma.
\end{equation}
Clearly,
\begin{equation}\label{changed-pdesolution0}
\mathcal T(t)=T_0\mathrm{e}^{\sigma t},
\end{equation}
where $T_0=\mathcal T(0)$ is a constant to be determined.

To solve for $\mathcal X(x)$, we must consider three cases: (i) $c<\sigma$; (ii) $c=\sigma$; (iii) $c>\sigma$.
In cases (i) and (ii), the general solutions of (\ref{definitionY1}) are, respectively,
\begin{equation*}
  \mathcal X(x)=X_0\mathrm{e}^{\sqrt{\sigma-c} x}+X_1 \mathrm{e}^{-\sqrt{\sigma-c} x},
\end{equation*}
and
\begin{equation*}
  \mathcal X(x)=X_0+X_1 x,
\end{equation*}
where $X_0$ and $X_1$ are constants to be determined from the boundary conditions (\ref{definition-inY}) and (\ref{definition-inY1}).
Then the corresponding solutions of (\ref{closed-pde}) are
\begin{equation*}
  y(x,t)=X_0T_0\mathrm{e}^{\sqrt{\sigma-c} x+\sigma t}+X_1 T_0 \mathrm{e}^{-\sqrt{\sigma-c} x+\sigma t},
\end{equation*}
and
\begin{equation*}
  y(x,t)=X_0T_0\mathrm{e}^{\sigma t}+X_1T_0 x \mathrm{e}^{\sigma t}.
\end{equation*}
These solutions are clearly unstable because $0<c\leq\sigma$. Thus, we want to choose the parameters $\theta_1$ and $\theta_2$ so that the unique solution of (\ref{closed-pde}) satisfies case (iii) instead of cases (i) and (ii).
%

In case (iii), the general solution of (\ref{definitionY1}) is
\begin{equation}\label{changed-pdesolution}
\mathcal X(x)=X_0\cos(\sqrt{c-\sigma} x)+X_1\sin(\sqrt{c-\sigma} x),
\end{equation}
where $X_0$ and $X_1$ are  constants to be determined from the boundary conditions (\ref{definition-inY}) and (\ref{definition-inY1}).
Substituting (\ref{changed-pdesolution}) into (\ref{definition-inY}), we obtain
\begin{equation}
\mathcal X(0)=X_0=0.
\end{equation}
Hence,
\begin{equation} \label{definX}
\mathcal X(x)=X_1\sin(\sqrt{c-\sigma}x).
\end{equation}
To simplify the notation, we introduce a new variable $\alpha=\sqrt{c-\sigma}$.
Substituting (\ref{definX}) into condition (\ref{definition-inY1}), we have
\begin{equation*}\label{definition-x11}
\begin{aligned}
X_1\sin\alpha=X_1\int_0^1\theta_1 \xi\sin(\alpha \xi)d\xi+X_1\int_0^1\theta_2 \xi^2\sin(\alpha \xi)d\xi,
\end{aligned}
\end{equation*}
and thus
\begin{equation}\label{definition-x1}
\begin{aligned}
\sin\alpha=\int_0^1\theta_1 \xi\sin(\alpha \xi)d\xi+\int_0^1\theta_2 \xi^2\sin(\alpha \xi)d\xi.
\end{aligned}
\end{equation}
Evaluating the first integral on the right hand side of (\ref{definition-x1}) gives
\begin{equation}\label{integal1}
\begin{aligned}
\int_0^1\theta_1 \xi\sin(\alpha \xi)d\xi&
=\theta_1\left(\frac{\sin\alpha}{\alpha^2}-\frac{\cos\alpha}{\alpha}\right).
\end{aligned}
\end{equation}
Evaluating the second integral on the right hand side of (\ref{definition-x1}) gives
\begin{equation}\label{integal2}
\begin{aligned}
\int_0^1\theta_2 \xi^2\sin(\alpha \xi)d\xi
&=-\theta_2\left(\frac{\cos\alpha}{\alpha}-\frac{2\sin\alpha}{\alpha^2}\right)+\frac{2\theta_2(\cos\alpha-1)}{\alpha^3}.\\
\end{aligned}
\end{equation}
Thus, using (\ref{integal1}) and (\ref{integal2}), (\ref{definition-x1}) can be simplified as
\begin{equation}\label{change-lambda-theda1-theda2}
\begin{aligned}
(\theta_1\alpha^2+\theta_2\alpha^2-2\theta_2)\cos\alpha+(\alpha^3-\theta_1\alpha-2\theta_2\alpha)\sin\alpha+2\theta_2=0.
\end{aligned}
\end{equation}
The following result, the proof of which is deferred to the appendix, is fundamental to our subsequent analysis.
\begin{lemma}\label{lemma1}
Suppose $\Theta=(\theta_1,\theta_2)^\top$ satisfies the following inequality:
\begin{equation}\label{constraint00}
\theta_1^2+\theta_2^2+2\theta_1\theta_2-2\theta_1-4\theta_2\geq0.
\end{equation}
Then equation (\ref{change-lambda-theda1-theda2}) has an infinite number of positive solutions.
\end{lemma}

For any  $\alpha$ satisfying (\ref{change-lambda-theda1-theda2}), there exists a corresponding solution of (\ref{definitionY1}) in the form (\ref{definX}). Let $\{\alpha_n\}_{n=1}^\infty$ be a sequence of positive solutions of (\ref{change-lambda-theda1-theda2}). Then the general solution of (\ref{definitionY1}) is
\begin{equation}\label{xsolution}
\begin{aligned}
  \mathcal X(x)&=\sum_{n=1}^\infty A_n \sin(\alpha_n x),\\
\end{aligned}
\end{equation}
where $A_n$ are constants to be determined. The corresponding eigenvalues are
\begin{equation}\label{eigenvalues0}
  \sigma_n=c- \alpha_n^2, \quad n=1,2,3,\dots
\end{equation}
Hence, using (\ref{changed-pdesolution0}),
\begin{equation}\label{ysolution0}
\begin{aligned}
  y(x,t)=\sum_{n=1}^\infty T_0 A_n\mathrm{e}^{(c- \alpha_n^2) t} \sin(\alpha_n x).
\end{aligned}
\end{equation}
By virtue of (\ref{definition-inY}) and (\ref{definition-inY1}), this solution satisfies the boundary conditions (\ref{closed-pdeifl}) and (\ref{cpdefr11}). The constants $T_0$ and $A_n$  must be selected appropriately so that the initial condition (\ref{closed-pdeini}) is also satisfied. To ensure stability as $t\rightarrow \infty$, each eigenvalue $\sigma_n=c- \alpha_n^2$ in (\ref{ysolution0}) must be negative. Thus, 
we impose the following constraints on $\Theta=(\theta_1,\theta_2)^\top$:
%
%
\begin{subequations}\label{constraints}
\begin{align}
&\theta_1^2+\theta_2^2+2\theta_1\theta_2-2\theta_1-4\theta_2\geq0,\label{constraint0} \\
&c- \alpha^2\leq-\epsilon, \label{constraint1} \\
&(\theta_1\alpha^2+\theta_2\alpha^2-2\theta_2)\cos\alpha+(\alpha^3-\theta_1\alpha-2\theta_2\alpha)\sin\alpha+2\theta_2=0, \label{constraint2}
\end{align}
\end{subequations}
where $\epsilon$ is a given positive parameter and $\alpha$ is the smallest positive solution of (\ref{change-lambda-theda1-theda2}). Note that $\alpha$ here is treated as an additional optimization variable. Constraint (\ref{constraint0}) ensures that there are an infinite number of eigenvalues (see Lemma \ref{lemma1}) and thus the solution form (\ref{ysolution0}) is valid. Constraints (\ref{constraint1}) and (\ref{constraint2}) ensure that the largest eigenvalue is negative, thus guaranteeing solution stability. Adding constraints (\ref{constraints}) to Problem P$_1$ yields the following modified problem.
\begin{problem2}
Given the PDE system (\ref{closed-pde}) with the parameterized kernel (\ref{kequation}), choose $\Theta=(\theta_1,\theta_2)^\top$ and $\alpha$ such that the cost functional (\ref{optcostfun}) is minimized subject to the bound constraints (\ref{kernelbound}) and the nonlinear constraints (\ref{constraints}).
\end{problem2}

The next result is concerned with the stability of the closed-loop system corresponding to the optimized kernel from Problem P$_2$.
\begin{theorem}\label{theorem0}
Let $(\Theta^\ast,\alpha^\ast)$ be an optimal solution of Problem P$_2$, where $\alpha^\ast$ is the smallest positive solution of equation (\ref{constraint2}) corresponding to $\Theta^\ast$. Suppose that there exists a sequence $\{\alpha_n^\ast\}_{n=1}^\infty$ of positive solutions to equation (\ref{constraint2}) corresponding to $\Theta^\ast$ such that $y_0(x)\in \mathrm{span}\{\sin(\alpha_n^\ast x)\}$. Then the closed-loop system (\ref{closed-pde}) corresponding to $\Theta^\ast$ is stable.
\end{theorem}
\begin{proof}
~Because of constraint (\ref{constraint0}), the solution ~form ~(\ref{ysolution0}) ~with $\alpha_n=\alpha_n^\ast$ ~is ~guaranteed ~to ~satisfy ~(\ref{closed-pdef}), (\ref{closed-pdeifl}) and (\ref{cpdefr11}). If $y_0(x)\in \mathrm{span}\{\sin(\alpha_n^\ast x)\}$, then there exists  constants $Y_n, n\geq 1$, such that
\begin{equation*}
\begin{aligned}
   y_0(x)&=\sum_{n=1}^\infty Y_n \sin(\alpha^\ast_n x).
\end{aligned}
\end{equation*}
Taking $Y_n=T_0A_n$ ensures that (\ref{ysolution0})  with $\alpha_n=\alpha_n^\ast$ also satisfies the initial conditions (\ref{closed-pdeini}), and is therefore the unique solution of (\ref{closed-pde}). Since $\alpha^\ast$ is the first positive solution of equation (\ref{constraint2}), it follows from constraint (\ref{constraint1}) that for each $n \geq 1$,
\begin{equation*}
  c-(\alpha_n^\ast)^2\leq c-(\alpha^\ast)^2\leq -\epsilon <0.
\end{equation*}
This shows that all eigenvalues are negative.
\end{proof}
Theorem \ref{theorem0} requires that the initial function $y_0(x)$ be contained within the linear span of sinusoidal  functions $\sin(\alpha_n^\ast x)$, where each
$\alpha_n^\ast$ is a solution of equation (\ref{change-lambda-theda1-theda2})  corresponding to $\Theta^\ast$. The good thing about this condition is that it can be verified numerically by solving the following optimization problem:
choose span coefficients $Y_n, 1\leq n \leq N$, to minimize
\be \label{cosfunctionJ}
 J=\int_0^1\bigg|y_0(x)-\sum_{n=1}^N Y_n\sin(\alpha_n^\ast x)\bigg|^2 dx,
\ee
where $N$ is a sufficiently large integer and each $\alpha_n^\ast$ is a solution of equation (\ref{change-lambda-theda1-theda2}) corresponding to the optimal solution of Problem P$_2$. If the optimal cost value for this optimization problem is sufficiently small, then the span condition in Theorem \ref{theorem0} is likely to be satisfied, and therefore  closed-loop stability is guaranteed.

Based on our computational experience, the span condition in Theorem \ref{theorem0} is usually satisfied. This can be explained as follows. In the proof of Lemma \ref{lemma1} (see the appendix), we show that for any $\epsilon\in (0,\frac{1}{2}\pi)$,  there exists at least one solution of (\ref{change-lambda-theda1-theda2}) in the interval $[k\pi-\epsilon,k\pi+\epsilon]$ when $k$ is sufficiently large. It follows that $k\pi$ is an approximate solution of (\ref{change-lambda-theda1-theda2}) for all sufficiently large $k$---in a sense, the solutions $\alpha_n^\ast$ of (\ref{change-lambda-theda1-theda2}) converge to the integer multiples of $\pi$. In our computational experience, this convergence occurs very rapidly. Thus, it is reasonable to expect that the linear span of $\{\sin(\alpha_n^\ast x)\}$ is ``approximately'' the same as the linear span of $\{\sin(n\pi x)\}$, which is known to be a basis for the space of continuous functions defined on $[0,1]$. 

\section{Gradient Computation} \label{sec:VariationalAnalysis}
Problem P$_2$ is an optimal parameter selection problem with decision parameters $\theta_1, \theta_2$  and $\alpha$. In principle, such problems can be solved as nonlinear optimization problems using the Sequential Quadratic Programming (SQP)
method or other nonlinear optimization methods. However, to do this, we need the gradients of the cost functional (\ref{optcostfun}) and the constraint functions (\ref{constraints}) with respect to the decision parameters.
The gradients of the constraint functions can be easily derived using elementary differentiation.
Define
\begin{equation*}
\begin{aligned}
&g_1(\Theta)=\theta_1^2+\theta_2^2+2\theta_1\theta_2-2\theta_1-4\theta_2,\nonumber \\
&g_2(\alpha)=c- \alpha^2, \nonumber \\
&g_3(\Theta,\alpha)=(\theta_1\alpha^2+\theta_2\alpha^2-2\theta_2)\cos\alpha+(\alpha^3-\theta_1\alpha-2\theta_2\alpha)\sin\alpha+2\theta_2. \nonumber
\end{aligned}
\end{equation*}
Then the corresponding constraint gradients  are given by
\begin{equation}\label{g0-dp-in0}
\begin{aligned}
\nabla_{\theta_1} g_1(\Theta)=2\theta_1+2\theta_2-2,\quad \nabla_{\theta_2}g_1(\Theta)=2\theta_2+2\theta_1-4,\quad
\nabla_\alpha g_2(\alpha)=-2\alpha,   
\end{aligned}
\end{equation}
and
\begin{subequations}\label{g0-dp-eq}
\begin{align}
\nabla_{\theta_1} g_3(\Theta,\alpha)&=\alpha^2\cos\alpha-\alpha\sin\alpha,\\ 
\nabla_{\theta_2} g_3(\Theta,\alpha)&=(\alpha^2-2)\cos\alpha-2\alpha\sin\alpha+2,\\ 
\nabla_\alpha g_3(\Theta,\alpha)&=(\alpha^3+\theta_1\alpha)\cos\alpha+(3\alpha^2-\theta_1\alpha^2-\theta_2\alpha^2-\theta_1)\sin\alpha. 
\end{align}
\end{subequations}
Since the constraint functions in (\ref{constraints}) are explicit functions of the decision variables, their gradients are easily obtained. The cost functional (\ref{optcostfun}), on the other hand, is an implicit function of $\Theta$ because it depends on the state trajectory $y(x,t)$. Thus, computing the gradient of (\ref{optcostfun}) is a non-trivial task. We now develop a computational method, analogous to the costate method in the optimal control of ordinary differential equations~\cite{lin2013optimal, linsurvey2013, Teo1991}, for computing this gradient.

We define the following costate PDE system:
\begin{subequations}\label{costate}
\begin{numcases}{}
v_t(x,t) + v_{xx}(x,t) + cv(x,t)+y(x,t;\Theta)-k(x;\Theta) v_x(1,t) =0,\\
v(0,t)=v(1,t)=0,\label{costatebd} \\
v(x,T)=0. \label{costateinition}
\end{numcases}
\end{subequations}
Let $v(x,t;\Theta)$ denote the solution of the costate PDE system (\ref{costate}) corresponding to the parameter vector $\Theta$. Then we have the following theorem.
\begin{theorem}
The gradient of the cost functional (\ref{optcostfun}) is given by
\begin{subequations}\label{g0-dp-0}
\begin{align}
\nabla_{\theta_1} g_0(\Theta)
&=-\int_0^T\int_0^1 xv_x(1,t;\Theta)y(x,t;\Theta)dxdt+\frac{1}{3}\theta_1+\frac{1}{4}\theta_2,  \\ \label{gradient1}
\nabla_{\theta_2} g_0(\Theta)
&=-\int_0^T\int_0^1 x^2v_x(1,t;\Theta)y(x,t;\Theta)dxdt+\frac{1}{4}\theta_1+\frac{1}{5}\theta_2.
\end{align}
\end{subequations}
\end{theorem}
\begin{proof}
Let $\nu(x,t)$ be an arbitrary function satisfying
\begin{equation}\label{setwfun}
\begin{aligned}
&\nu(x,T)=0,\quad \nu(0,t)=\nu(1,t)=0.
\end{aligned}
\end{equation}
Then we can rewrite the cost functional (\ref{optcostfun}) in augmented form as follows:
\begin{align}\label{g0}
g_0(\Theta)&=\frac{1}{2}\int_0^T\int_0^1 y^2(x,t;\Theta)dxdt+\frac{1}{2}\int_0^1k^2(x;\Theta)dx \nonumber\\
&\quad+\int_0^T\int_0^1\nu(x,t)\big\{-y_t(x,t;\Theta)+y_{xx}(x,t;\Theta)+c y(x,t;\Theta)\big\}dxdt.
\end{align}
Using integration by parts and applying the boundary condition (\ref{closed-pdeifl}), we can simplify the augmented cost functional (\ref{g0}) to obtain

\begin{align}\label{g0inpart}
g_0(\Theta)
&=\frac{1}{2}\int_0^T\int_0^1 y^2(x,t;\Theta)dxdt+\frac{1}{2}\int_0^1k^2(x;\Theta)dx \nonumber\\
&\quad-\int_0^1 \nu(x,T)y(x,T;\Theta)dx + \int_0^1 \nu(x,0)y(x,0)dx \nonumber\\
&\quad+\int_0^T\int_0^1 \nu_t(x,t) y(x,t;\Theta)dxdt+\int_0^T \big [\nu(x,t)y_x(x,t;\Theta)\big]_{x=0}^{x=1}dt \nonumber \\
&\quad-\int_0^T \nu_x(1,t)y(1,t;\Theta)dt+\int_0^T\int_0^1 \nu_{xx}(x,t) y(x,t;\Theta)dxdt \nonumber \\
&\quad+c \int_0^T\int_0^1\nu(x,t) y(x,t;\Theta)dxdt.
\end{align}
Thus, recalling the conditions (\ref{closed-pdeini}) and (\ref{setwfun}), we obtain
\begin{align}
\label{g0-1}
g_0(\Theta)&=\frac{1}{2}\int_0^T\int_0^1 \nonumber y^2(x,t;\Theta)dxdt+\frac{1}{2}\int_0^1k^2(x;\Theta)dx\\ \nonumber
&\quad+\int_0^T\int_0^1\big\{\nu_t(x,t) + \nu_{xx}(x,t) + c\nu(x,t)\big\}y(x,t;\Theta)dxdt \nonumber\\
&\quad+\int_0^1\nu(x,0)y_0(x)dx-\int_0^T\nu_x(1,t)y(1,t;\Theta)dt.\nonumber
\end{align}
Now, consider a perturbation $\varepsilon\rho$ in the parameter vector $\Theta$, where $\varepsilon$ is a constant of sufficiently small magnitude and $\rho$ is an arbitrary vector. The corresponding perturbation in the state is,
\be
\begin{aligned}\label{perturbY}
y(x,t;\Theta+\varepsilon\rho) = y(x,t;\Theta) +\varepsilon\langle \nabla_{\Theta} y(x,t;\Theta), \rho\rangle +\mathcal O(\varepsilon^2),
\end{aligned}
\ee
and the perturbation in the feedback kernel is,
\be
\begin{aligned}\label{perturbKernel}
k(x;\Theta+\varepsilon\rho) = k(x;\Theta) +\varepsilon\langle \nabla_{\Theta} k(x;\Theta), \rho\rangle +\mathcal O(\varepsilon^2),
\end{aligned}
\ee
where $\mathcal O(\varepsilon^2)$ denotes omitted second-order terms such that $\mathcal \varepsilon^{-1}\mathcal O(\varepsilon^2)\rightarrow 0$ as $\varepsilon \rightarrow 0$.
For notational simplicity, we define $\eta(x,t) = \langle \nabla_{\Theta} y(x,t;\Theta), \rho\rangle$. 
Obviously, $\eta(x,0)=0$, because the initial profile $y_0(x)$ is independent of the parameter vector $\Theta$.
Based on (\ref{perturbY}) and (\ref{perturbKernel}), the perturbed augmented cost functional takes the following form:
\begin{align}
\label{g0-p}
g_0(\Theta+\varepsilon\rho)&=\frac{1}{2}\int_0^T\int_0^1 \big\{y(x,t;\Theta)+\varepsilon\eta(x,t)\big\}^2dxdt\nonumber\\
&\quad+\int_0^T\int_0^1 \big\{\nu_t(x,t) + \nu_{xx}(x,t) + c\nu(x,t)\big\}\big\{y(x,t;\Theta)+\varepsilon\eta(x,t)\big\}dxdt\nonumber\\
&\quad+\int_0^1\nu(x,0)y_0(x)dx-\int_0^T\nu_x(1,t)\big\{y(1,t;\Theta)+\varepsilon\eta(1,t)\big\}dt\nonumber\\
&\quad+\frac{1}{2}\int_0^1\big\{k(x;\Theta) +\varepsilon\langle \nabla_{\Theta} k(x;\Theta), \rho\rangle\big\}^2dx+\mathcal O(\varepsilon^2).
\end{align}
From the boundary condition in (\ref{cpdefr11}), we have
\begin{align}
\label{g0-bound}
y(1,t;\Theta)+\varepsilon\eta(1,t)&=\int_0^1k(x;\Theta)\big\{y(x,t;\Theta)+\varepsilon\eta(x,t)\big\}dx \nonumber\\
&\quad+\int_0^1\varepsilon\langle\nabla_\Theta k(x;\Theta),\rho\rangle y(x,t;\Theta) dx+\mathcal O(\varepsilon^2).
\end{align}
Substituting (\ref{g0-bound}) into (\ref{g0-p}) gives
\begin{align}
\label{g0-p1}
g_0(\Theta+\varepsilon\rho)&=\frac{1}{2}\int_0^T\int_0^1 \big\{y(x,t;\Theta)+\varepsilon\eta(x,t)\big\}^2dxdt\nonumber\\
&\quad+\int_0^T\int_0^1 \big\{\nu_t(x,t) + \nu_{xx}(x,t) + c\nu(x,t)\big\}\big\{y(x,t;\Theta)+\varepsilon\eta(x,t)\big\}dxdt\nonumber\\
&\quad+\int_0^1\nu(x,0)y_0(x)dx-\int_0^T\nu_x(1,t)\left[\int_0^1k(x;\Theta)\big\{y(x,t;\Theta)+\varepsilon\eta(x,t)\big\}dx\right]dt \nonumber \\
&\quad-\int_0^T\nu_x(1,t)\left[\int_0^1\varepsilon\langle\nabla_\Theta k(x;\Theta),\rho\rangle y(x,t;\Theta)dx\right]dt  \nonumber\\
&\quad+\frac{1}{2}\int_0^1\big\{k(x;\Theta) +\varepsilon\langle \nabla_{\Theta} k(x;\Theta), \rho\rangle\big\}^2dx+\mathcal O(\varepsilon^2).
\end{align}
Taking the derivative of (\ref{g0-p1}) with respect to $\varepsilon$ and setting $\varepsilon=0$ gives

\begin{align}
\label{g0-dp}
\langle\nabla_\Theta g_0(\Theta), \rho\rangle&=\left.\frac{d g_0(\Theta+\varepsilon\rho)}{d\varepsilon}\right|_{\varepsilon=0} \nonumber \\
&=\int_0^T\int_0^1 \big\{y(x,t;\Theta)+\nu_t(x,t) + \nu_{xx}(x,t) + c\nu(x,t)\big\}\eta(x,t) dxdt \nonumber\\
&\quad-\int_0^T\int_0^1\nu_x(1,t)k(x;\Theta) \eta(x,t)dxdt \nonumber\\
&\quad-\int_0^T\int_0^1\nu_x(1,t)\langle\nabla_\Theta k(x;\Theta),\rho\rangle y(x,t;\Theta) dxdt \nonumber\\
&\quad+\int_0^1k(x;\Theta)\langle \nabla_{\Theta} k(x;\Theta), \rho\rangle dx. 
\end{align}
Choosing the multiplier $\nu(x,t)$ to be the solution of the costate system
(\ref{costate}),
the gradient in (\ref{g0-dp}) becomes
\be
\begin{aligned}\nonumber
\langle\nabla_\Theta g_0(\Theta), \rho\rangle
=&-\int_0^T\int_0^1v_x(1,t;\Theta)\langle\nabla_\Theta k(x;\Theta),\rho\rangle y(x,t;\Theta) dxdt\\
&+\int_0^1k(x;\Theta)\langle \nabla_{\Theta} k(x;\Theta), \rho\rangle dx .
\end{aligned}
\ee
Taking $\rho=(1,0)^\top$ gives
\begin{equation*}
\nabla_{\theta_1} g_0(\Theta)=-\int_0^T\int_0^1 xv_x(1,t;\Theta)y(x,t;\Theta)dxdt+\frac{1}{3}\theta_1+\frac{1}{4}\theta_2.
\end{equation*}
Similarly, taking $\rho=(0,1)^\top$ gives
\begin{equation*}
\nabla_{\theta_2} g_0(\Theta)=-\int_0^T\int_0^1 x^2v_x(1,t;\Theta)y(x,t;\Theta)dxdt+\frac{1}{4}\theta_1+\frac{1}{5}\theta_2.
\end{equation*}
This completes the proof.
\end{proof}
\section{Numerical Solution Procedure}\label{sec:NumericalAlgorithms}
Based on the gradient formulas derived in Section \ref{sec:VariationalAnalysis}, we now propose a gradient-based optimization framework for solving Problem P$_2$. This framework is illustrated in Figure \ref{fig:process} and described in detail below.
\begin{algorithm}\label{Algorithm2}
Gradient-based optimization procedure for solving Problem P$_2$.
\textsf{\begin{enumerate}
  \item Choose an initial guess $(\theta_{1}, \theta_{2}, \alpha)$.
  \item Solve the state PDE system (\ref{closed-pde}) corresponding to $(\theta_{1}, \theta_{2})$. \label {step2}
  \item Solve the costate PDE system (\ref{costate}) corresponding to $(\theta_{1}, \theta_{2})$.
  \item Compute the cost and constraint gradients at $(\theta_{1}, \theta_{2}, \alpha)$ using (\ref{g0-dp-in0}), (\ref{g0-dp-eq}), and (\ref{g0-dp-0}).
  \item Use the gradient information obtained in Step (d) to perform an optimality test. If $(\theta_{1}, \theta_{2}, \alpha)$ is optimal, then stop; otherwise, go
to Step (f). \label{performtest}
  \item Use the gradient information obtained in Step (d) to calculate a search direction. \label{step5}
  \item Perform a line search to determine the optimal step length.
  \item Compute a new point $(\theta_{1}, \theta_{2}, \alpha)$ and return to Step (b). \label{endtest}
\end{enumerate}}
\end{algorithm}
Note that Steps (e)-(h) of Algorithm \ref{Algorithm2} can be performed automatically by standard nonlinear optimization solvers
such as FMINCON in MATLAB.

Recall from Theorem \ref{theorem0} that to guarantee closed-loop stability, the optimal value of $\alpha$ must be the \emph{first} positive solution of (\ref{change-lambda-theda1-theda2}). In practice, this can usually be achieved by choosing $\alpha=0$ as the initial guess. Moreover, after solving Problem P$_2$, it is easy to check whether the optimal value of $\alpha$ is indeed the smallest positive solution by plotting the left-hand side of (\ref{change-lambda-theda1-theda2}).
\begin{figure}
\begin{center}
\includegraphics[width=4in]{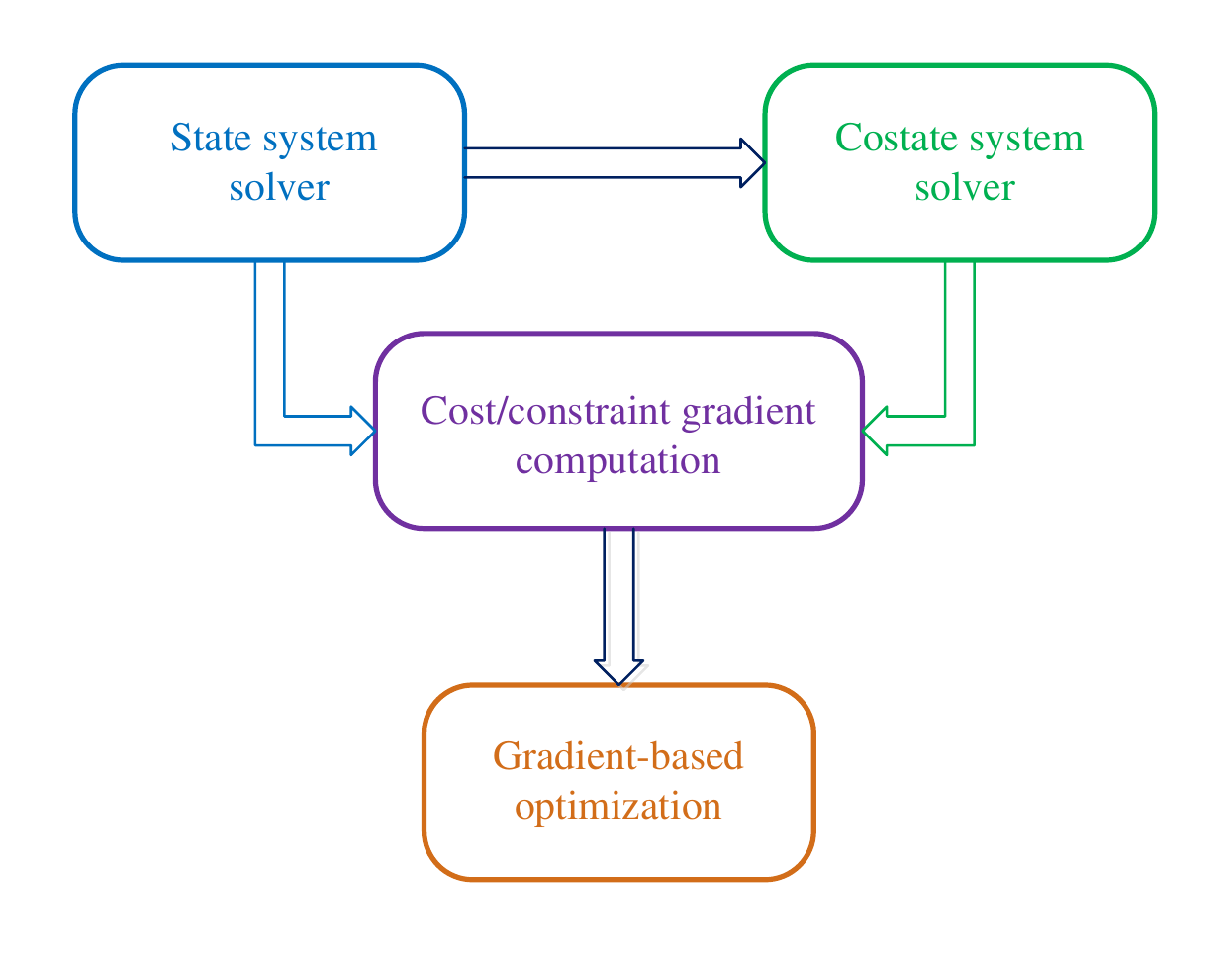}
\caption{Gradient-based optimization framework for solving Problem P$_2$.}\label{fig:process}
\end{center}
\end{figure}

\subsection{Simulation of the state system}\label{SecClosed}
To solve the state system (\ref{closed-pde}) numerically, we will develop a finite-difference method. This method involves discretizing both the spatial and the temporal domains into a finite number of subintervals, i.e.,
\begin{subequations}
\begin{align}
&x_0 = 0,\, x_1 = h,\, x_2 = 2h,\ldots,x_n=nh = 1,\\
&t_0 = 0,\, t_1 = \tau,\,t_2 = 2\tau,\ldots,t_m = m\tau=T,
\end{align}
\end{subequations}
where $n$ and $m$ are positive integers and $h=1/n$ and $\tau=T/m$.
Using the Taylor expansion, we obtain the following approximations:
\begin{subequations}\label{discretize_y}
\begin{align}
&\frac{\partial y(x_i,t_j)}{\partial t}=\frac{y(x_i,t_j+\tau)-y(x_i,t_j)}{\tau}+\mathcal O(\tau),\\
&\frac{\partial^2 y(x_i,t_j)}{\partial x^2}=\frac{y(x_i+h,t_j)-2y(x_i,t_j)+y(x_i-h,t_j)}{h^2}+\mathcal O(h^2),
\end{align}
\end{subequations}
where $\mathcal O(\tau)$ and $\mathcal O(h^2)$ denote, respectively, omitted first- and second-order terms such that $\mathcal O(\tau)\rightarrow 0$ as $\tau \rightarrow 0$ and $h^{-1} \mathcal O(h^2)\rightarrow 0$ as $h \rightarrow 0$. Substituting (\ref{discretize_y}) into (\ref{closed-pdef}) gives

\begin{equation}\label{discretize_y0}
\begin{aligned}
\frac{y_{i,j+1}-y_{i,j}}{\tau}=\frac{y_{i+1,j}-2y_{i,j}+y_{i-1,j}}{h^2}+cy_{i,j},
\end{aligned}
\end{equation}
where $y_{i,j}=y(x_i,t_j)$, $i=0,1,\ldots,n$, $j=0,1,\ldots,m$. Simplifying this equation, we obtain
\begin{equation}\label{discretize_y1}
\begin{aligned}
y_{i,j+1}=(1-2r+c\tau)y_{i,j}+r(y_{i-1,j}+y_{i+1,j}),
\end{aligned}
\end{equation}
where  $1\leq i \leq n-1$, $0\leq j \leq m-1$ and
\begin{equation}\label{discretize_r}
r=\frac{\tau}{h^2}.
\end{equation}
The explicit numerical scheme (\ref{discretize_y1}) is convergent when $0< r \leq 0.5$ (see reference \cite{burden1993} for the relevant convergence analysis). Thus, in this paper, we assume that $\tau$ and $h$ are chosen such that $0< r \leq 0.5$. From (\ref{closed-pdeini}), we obtain the initial condition
\begin{equation}\label{discretize_yini}
y_{i,0}=y(x_i,0)=y_0(x_i), ~~i=0, 1, \ldots, n.
\end{equation}
Moreover, from (\ref{closed-pdeifl}) and (\ref{cpdefr11}), we obtain the boundary conditions
\begin{equation}\label{discretize_ybd0}
y_{0,j}=y(0,t_j)=0,~~j=1,2,\ldots,m,
\end{equation}
and
\begin{equation}\label{y1tdic}
y_{n,j}=y(1,t_j)=\int_0^1 k(\xi;\Theta)y(\xi,t_j)d\xi, ~~j=1,2,\dots,m.
\end{equation}
Using the composite trapezoidal rule \cite{burden1993}, the integral in (\ref{y1tdic}) becomes

\begin{align}\label{discretize_ybd1}
y_{n,j}
&=\frac{1}{2}h\big\{k(x_0;\Theta)y(x_0,t_j)+k(x_n;\Theta)y(x_n,t_j)\big\}+h\sum_{i=1}^{n-1}k(x_i;\Theta)y(x_i,t_j)\nonumber\\
&=\frac{1}{2}hk(x_n;\Theta)y_{n,j}+h\sum_{i=1}^{n-1}k(x_i;\Theta)y_{i,j}.
\end{align}
Rearranging this equation yields
\begin{equation}\label{discretize_ybd11}
\begin{aligned}
y_{n,j}
&=\left[1-\frac{1}{2}hk(x_n;\Theta)\right]^{-1}\left[h\sum_{i=1}^{n-1}k(x_i;\Theta)y_{i,j}\right].
\end{aligned}
\end{equation}
By using the initial condition (\ref{discretize_yini}) and the boundary conditions (\ref{discretize_ybd0}) and (\ref{discretize_ybd11}),
numerical approximations of $y(x,t)$ at the pre-defined nodes can be calculated forward in time recursively from (\ref{discretize_y1}).

\subsection{Simulation of the costate system }\label{SecCos}

As with the state system, we will use the finite-difference method to solve the costate system (\ref{costate}) numerically. Using the Taylor expansion, we obtain the following approximations:
\begin{subequations}\label{discretize_v}
\begin{align}
\frac{\partial v(x_i,t_j)}{\partial t} &= \frac{v(x_i,t_j)-v(x_i,t_j-\tau)}{\tau}+\mathcal O(\tau),\\
\frac{\partial^2 v(x_i,t_j)}{\partial x^2} &= \frac{v(x_i+h,t_j)-2v(x_i,t_j)+v(x_i-h,t_j)}{h^2}+\mathcal O(h^2),\\
\frac{\partial v(1,t_j)}{\partial x} &= \frac{v(x_n,t_j)-v(x_n-h,t_j)}{h}+\mathcal O(h).
\end{align}
\end{subequations}
Substituting (\ref{discretize_v}) into (\ref{costate}) gives
\begin{equation}\label{discretize_v0}
\begin{aligned}
\frac{v_{i,j}-v_{i,j-1}}{\tau}+\frac{v_{i+1,j}-2v_{i,j}+v_{i-1,j}}{h^2}+cv_{i,j}+y_{i,j}-k(x_i;\Theta)\frac{v_{n,j}-v_{n-1,j}}{h}=0,
\end{aligned}
\end{equation}
where $v_{i,j}=v(x_i,t_j)$.
We rearrange this equation to obtain
\begin{equation}\label{discretize_v1}
\begin{aligned}
v_{i,j-1}&=(1-2r+c\tau)v_{i,j}+r(v_{i+1,j}+v_{i-1,j})+\tau y_{i,j}-\frac{\tau k(x_i;\Theta)}{h}(v_{n,j}-v_{n-1,j}),
\end{aligned}
\end{equation}
where $1\leq i \leq n-1$, $1\leq j \leq m$ and $r$ is as defined in (\ref{discretize_r}). From (\ref{costateinition}), we obtain the terminal condition
\begin{equation}\label{discretize_vini}
v_{i,m}=v(x_i,T)=0,~~i=0,1,\dots,n.
\end{equation}
Moreover, from (\ref{costatebd}), we obtain the boundary conditions
\begin{equation}\label{discretize_vbd}
v_{0,j}=v(0,t_j)=0,\ \ v_{n,j}=v(1,t_j)=0,~~j=0,1,\dots,m.
\end{equation}
Using the recurrence equation (\ref{discretize_v1}), together with (\ref{discretize_vini}) and (\ref{discretize_vbd}), we can compute approximate values of $v(x,t)$ backward in time. The finite-difference schemes for solving the state and costate PDEs are summarized in Table \ref{tab:test}.
\begin{table}
\begin{center}
 \caption{\label{tab:test}Numerical computation of  $y(x,t)$ and $v(x,t)$}
 \begin{tabular}{lcl}
  \toprule
   \textbf{Procedure 1}. Evaluation of $y(x_i,t_j)$.\\
  \midrule
\textbf{1}: Set $1/n\rightarrow h$, $T/m\rightarrow\tau$.\\
\textbf{2}: Compute $y_{i,0}$ for each $i=0,1,\dots,n$ using (\ref{discretize_yini}).\\ 
\textbf{3}: Set $1\rightarrow j$.\\
\textbf{4}: Compute $y_{i,j}$ for each $i=1,2,\dots,n-1$ by solving (\ref{discretize_y1}).\\
\textbf{5}: Compute $y_{0,j}$ using (\ref{discretize_ybd0}).\\
\textbf{6}: Compute $y_{n,j}$ using (\ref{discretize_ybd11}).\\
\textbf{7}: If $j=m$, then stop. Otherwise, set $j+1\rightarrow j$ and go to Step 4.\\
\toprule
  \textbf{Procedure 2}. Evaluation of $v(x_i,t_j).$\\
\midrule
\textbf{1}: Compute $v_{i,m}$ for each $i=0,1,\dots,n$ using (\ref{discretize_vini}).\\
\textbf{2}: Set $j=m-1$.\\ 
\textbf{3}: Compute $v_{i,j}$ for each  $i=1,2,\dots,n-1$ using (\ref{discretize_v1}).\\
\textbf{4}: Compute $v_{0,j}$ and $v_{n,j}$ using (\ref{discretize_vbd}).\\
\textbf{5}: If $j=0$, then stop. Otherwise, set $j-1\rightarrow j$ and go to Step 3.\\
  \bottomrule
 \end{tabular}
 \end{center}
\end{table}

\subsection{Numerical integration}
Recall the cost functional (\ref{optcostfun}):
\begin{equation*}
\begin{aligned}
g_0(\Theta)&=\frac{1}{2}\int_0^T\int_0^1 y^2(x,t;\Theta)dxdt+\frac{1}{2}\int_0^1k^2(x;\Theta)dx\\
&=\frac{1}{2}\int_0^T\int_0^1 y^2(x,t;\Theta)dxdt+\frac{1}{6}\theta_1^2+\frac{1}{10}\theta_2^2+\frac{1}{4}\theta_1\theta_2.
\end{aligned}
\end{equation*}
Furthermore, recall the cost functional's gradient from (\ref{g0-dp-0}):
\begin{equation*}
\begin{aligned}
\nabla_{\theta_1} g_0(\Theta)
&=-\int_0^T\int_0^1 xv_x(1,t;\Theta)y(x,t;\Theta)dxdt+\frac{1}{3}\theta_1+\frac{1}{4}\theta_2,  \\
\nabla_{\theta_2} g_0(\Theta)
&=-\int_0^T\int_0^1 x^2v_x(1,t;\Theta)y(x,t;\Theta)dxdt+\frac{1}{4}\theta_1+\frac{1}{5}\theta_2.
\end{aligned}
\end{equation*}
Clearly, both the cost functional (\ref{optcostfun}) and its gradient (\ref{g0-dp-0}) involve evaluating double integrals of the form
\begin{equation}\label{defindbinter}
\int_0^T\int_0^1\psi(x,t)dxdt,
\end{equation}
where $\psi(x,t)=y^2(x,t;\Theta)$ for the cost functional and $\psi(x,t)=-\nabla_{\theta_i} k(x;\Theta)v_x(1,t;\Theta)y(x,t;\Theta)$, $i=1,2$, for the cost functional's  gradient. To evaluate these integrals, we partition the space and temporal domains using the same equally-spaced mesh points $x_0, x_1, \ldots, x_n$ and $t_0, t_1, \ldots, t_m$ as in Sections \ref{SecClosed} and \ref{SecCos}.  These subintervals define step sizes $h=1/n$ and $\tau=T/m$. The integral in (\ref{defindbinter}) can be written as the following iterated integral:
\begin{equation}\label{defindbinter1}
\int_0^T\int_0^1\psi(x,t)dxdt=\int_0^T\left(\int_0^1\psi(x,t)dx\right)dt.
\end{equation}
Applying the composite Simpson's rule~\cite{burden1993} twice, we obtain the following approximation:
\begin{align}\label{dbintergrad}
 \int_0^T\int_0^1\psi(x,t)dxdt &= \frac{1}{3}\tau \bigg\{\phi(t_0)+2\sum\limits_{l = 1}^{(m/2) - 1}\phi(t_{2l})+4\sum\limits_{l = 1}^{m/2}\phi(t_{2l-1})+\phi(t_m)\bigg\},
\end{align}
where
\begin{equation*}
  \phi(t_j)=\frac{1}{3}h\bigg[\psi(x_0,t_j ) + 2\sum\limits_{k = 1}^{(n/2) - 1} {\psi(x_{2k} ,t_j )} +{4\sum\limits_{k = 1}^{n/2} {\psi(x_{2k - 1} ,t_j ) + \psi(x_n ,t_j )}}\bigg].
\end{equation*}
More details on numerical integration algorithms are available in \cite{burden1993}.
Using (\ref{dbintergrad}), the cost functional (\ref{optcostfun}) and its gradient (\ref{g0-dp-0}) can be evaluated successfully.

\section{Numerical Simulations}\label{sec:Numerical Simulations}
Our numerical simulations were conducted within the MATLAB programming environment running on a desktop computer with the following configuration: Intel Core i7-2600 3.40GHz CPU, 4.00GB RAM, 64-bit Windows 7 Operating System. For the finite-difference discretization, we used $n=14$ spatial intervals and $m=5000$ temporal intervals over a time horizon of $[0,T]=[0,4]$. Our code implements the gradient-based optimization procedure in Algorithm \ref{Algorithm2} by combining the FMINCON function in MATLAB with the gradient computation method described in Section \ref{sec:VariationalAnalysis}.

Consider the uncontrolled version of (\ref{closed-pde}) in which $u(t)=0$. In this case, the exact solution is
\begin{equation}\label{yexactsolution}
y(x,t)=2\sum_{n=1}^\infty C_n\mathrm{e}^{(c-n^2\pi^2)t}\sin(n \pi x)dx,
\end{equation}
where $C_n$ are the Fourier coefficients defined by
\begin{equation*}
C_n=\int_0^1y_0(x)\sin(n \pi x)dx.
\end{equation*}
The eigenvalues of (\ref{yexactsolution}) are $c-n^2\pi^2$, $n=1,2,\ldots$ The largest eigenvalue is therefore $c-\pi^2$, which indicates that system (\ref{closed-pde}) with $u(t)=0$ is unstable for $c>\pi^2\approx 9.8696$. We report the numerical results from our algorithm for three different scenarios.
\subsection{Scenario 1}
\begin{table}
\begin{center}
\caption{Solutions of (\ref{change-lambda-theda1-theda2}) and corresponding optimal span coefficients $Y_n$ in (\ref{cosfunctionJ}) for Scenario 1.}\label{tab:alpha10}
\begin{tabular}{cccc}
\toprule
$n$ &$\alpha^\ast_n$  &$\alpha^\ast_n/\pi$  & $Y_n$\\
\toprule
1&  3.3486   &1.0658  & 1.0364 \\
2&  6.3838   &2.0320  &$-0.0915$  \\
3&  9.4952   &3.0224  &0.0505   \\
4&  12.6173  &4.0162  &$-0.0360$  \\
5&  15.7493  &5.0131  & 0.0268  \\
6&  18.8835  &6.0108  &$-0.0206$  \\
7&  22.0205  &7.0093  & 0.0161  \\
8&  25.1582  &8.0081  &$-0.0126$  \\
9&  28.2971  &9.0072  & 0.0098  \\
10& 31.4363  &10.0064 &$-0.0074$  \\
\toprule
\end{tabular}
\end{center}
\end{table}
For ~the ~first ~scenario, ~we ~choose ~$c=10$ ~and ~$y_0(x)=\sin(\pi x)$. ~~The  corresponding uncontrolled ~open-loop ~response (see equation (\ref{yexactsolution})) is shown in Figure~\ref{fig:Original U10}.  As we can see from Figure~\ref{fig:Original U10}, the state of the uncontrolled system grows as time increases. For the feedback kernel optimization, we suppose that the lower and upper bounds for the optimization parameters are $a_i=-10$ and $b_i=10$, respectively. We also choose $\epsilon=1$ in (\ref{constraint1}).
Starting from the initial guess $(\theta_{1},\theta_{2}, \alpha)=(-1.0, 2.0, 0)$, our program terminates after 23 iterations and 15.8358 seconds. The optimal cost ~value ~is $g_0=0.1712$ and ~the ~optimal solution of ~Problem P$_2$ is $(\theta^\ast_{1},\theta^\ast_{2}, \alpha^\ast)=(-1.0775, 0.5966, 3.3486)$. 

The spatial-temporal response of the controlled plant corresponding to $(\theta_1^\ast, \theta_2^\ast)$ is shown in
Figure~5.2(a).
The figure clearly shows that the controlled system (\ref{closed-pde}) with optimized parameters $(\theta^\ast_1, \theta^\ast_2)$ is stable.  The corresponding boundary control and kernel function are shown in Figures~5.2(b) and 5.2(c), respectively. 

Recall from Theorem \ref{theorem0} that closed-loop stability is guaranteed if $\alpha^\ast=3.3486$ is the first positive solution of equation (\ref{change-lambda-theda1-theda2}) and the initial function $y_0(x)$ is contained within the linear span of $\{\sin(\alpha_n^\ast x)\}$, where each $\alpha_n^\ast$ is a solution of equation (\ref{change-lambda-theda1-theda2}) corresponding to $(\theta_1^\ast,\theta_2^\ast)$. It is clear from Figure~5.2(d) that $\alpha^\ast$ is indeed the first positive solution of equation (\ref{change-lambda-theda1-theda2}). To verify the linear span condition, we use FMINCON in MATLAB to minimize (\ref{cosfunctionJ}) for $N=10$.
The first 10 positive solutions of (\ref{change-lambda-theda1-theda2}) corresponding to the optimal parameters
$\theta_1^\ast=-1.0775$ and $\theta_2^\ast=0.5966$ are given in Table \ref{tab:alpha10}. The optimal span coefficients that minimize (\ref{cosfunctionJ}) are also given.  ~~The optimal value of $J$ in (\ref{cosfunctionJ}) is $2.184832\times 10^{-5}$, which indicates that the span condition holds. ~~Note ~also ~from ~Table \ref{tab:alpha10} that $\alpha_n^\ast/\pi$ converges to an integer as $n \rightarrow \infty$ (recall the discussion of the end of Section 2.2).



\begin{figure}
\begin{center}
\includegraphics[width=3.8in]{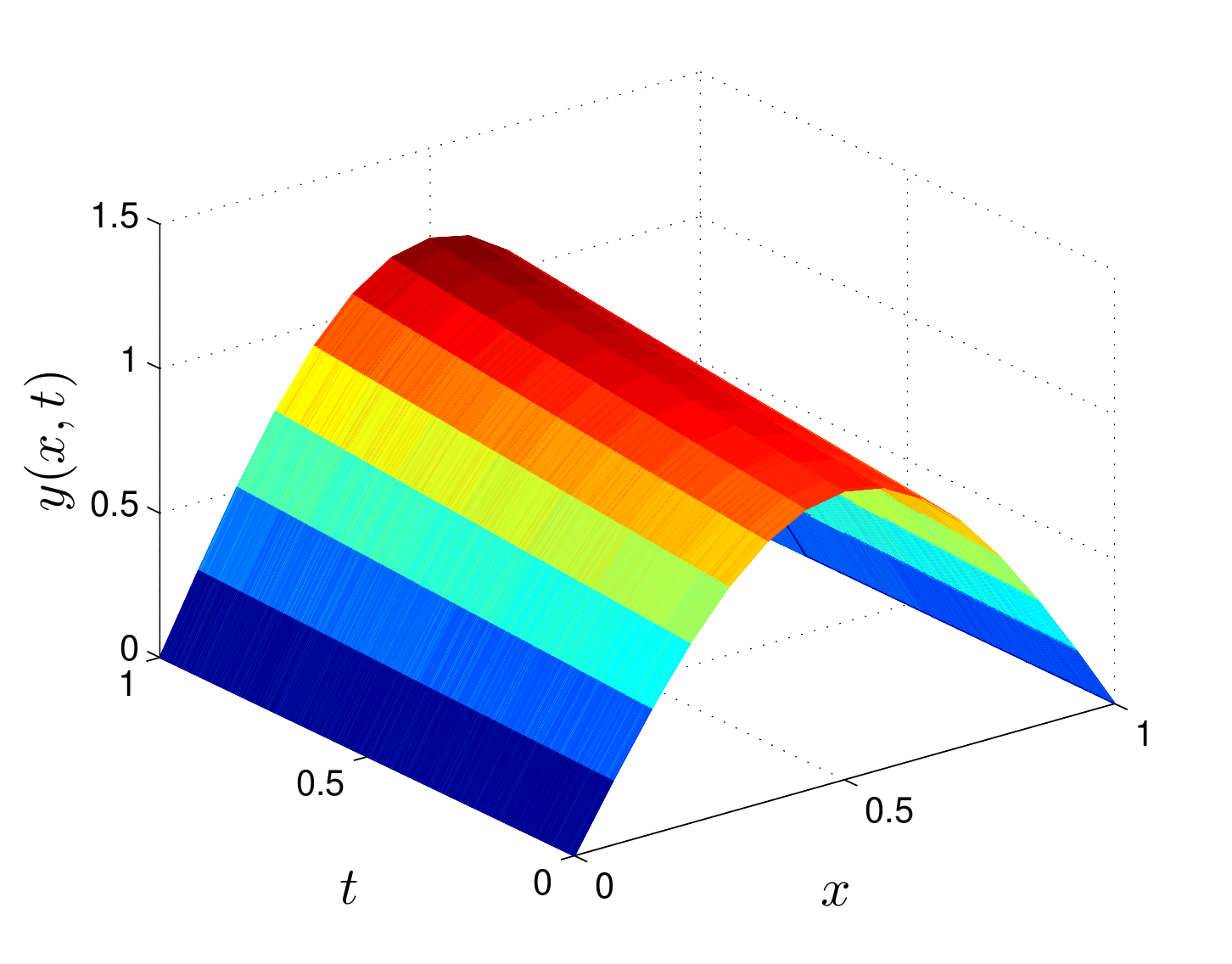}
\caption{Uncontrolled open-loop response  for Scenario 1.}\label{fig:Original U10}
\end{center}
\end{figure}

\begin{figure}\label{fig:U10}
\subfigure[Closed-loop response $y(x,t)$.] {\includegraphics[width=0.5\textwidth]{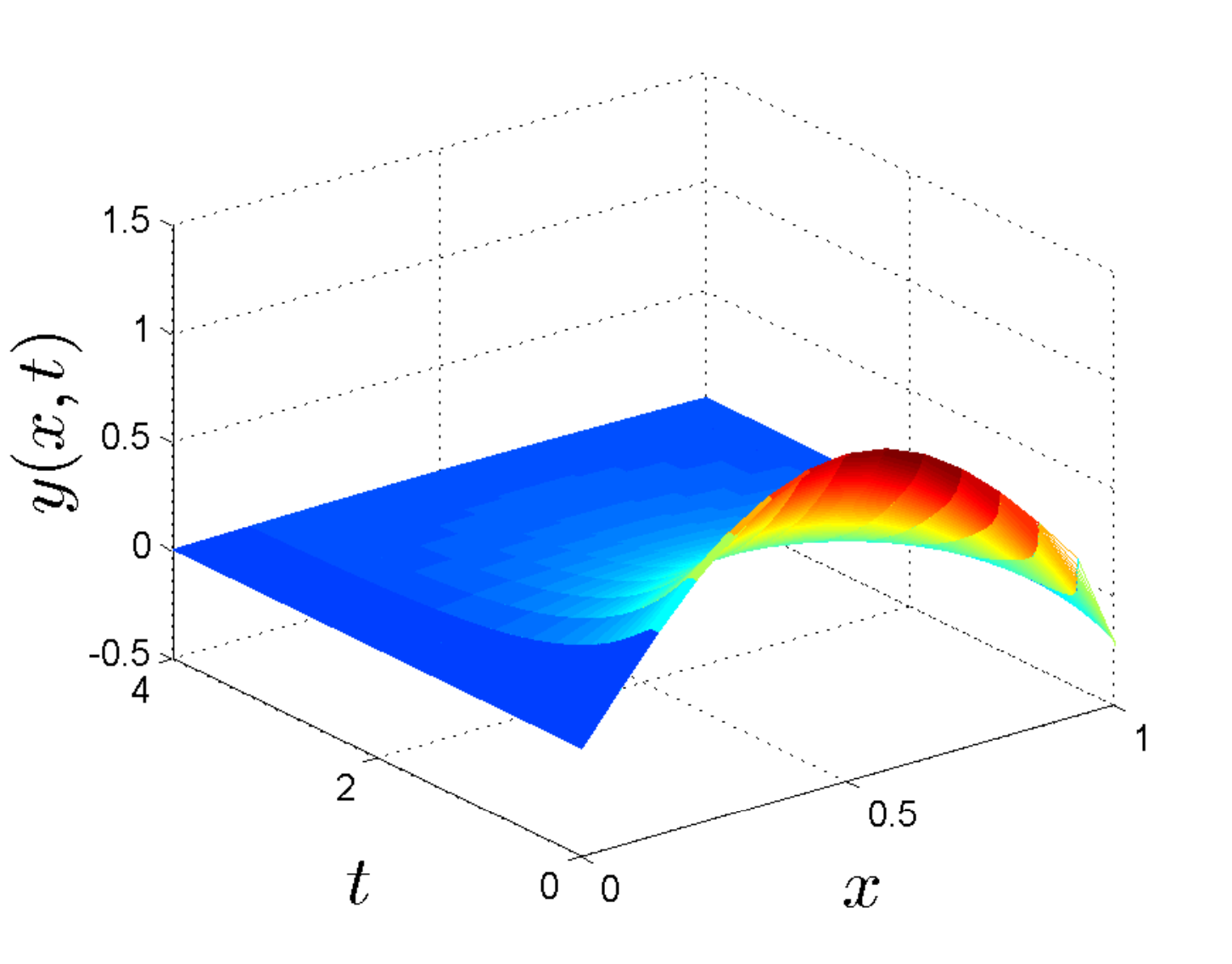}}\label{fig:U-10}
\subfigure[Optimal boundary control $y(1,t)$.] {\includegraphics[width=0.5\textwidth]{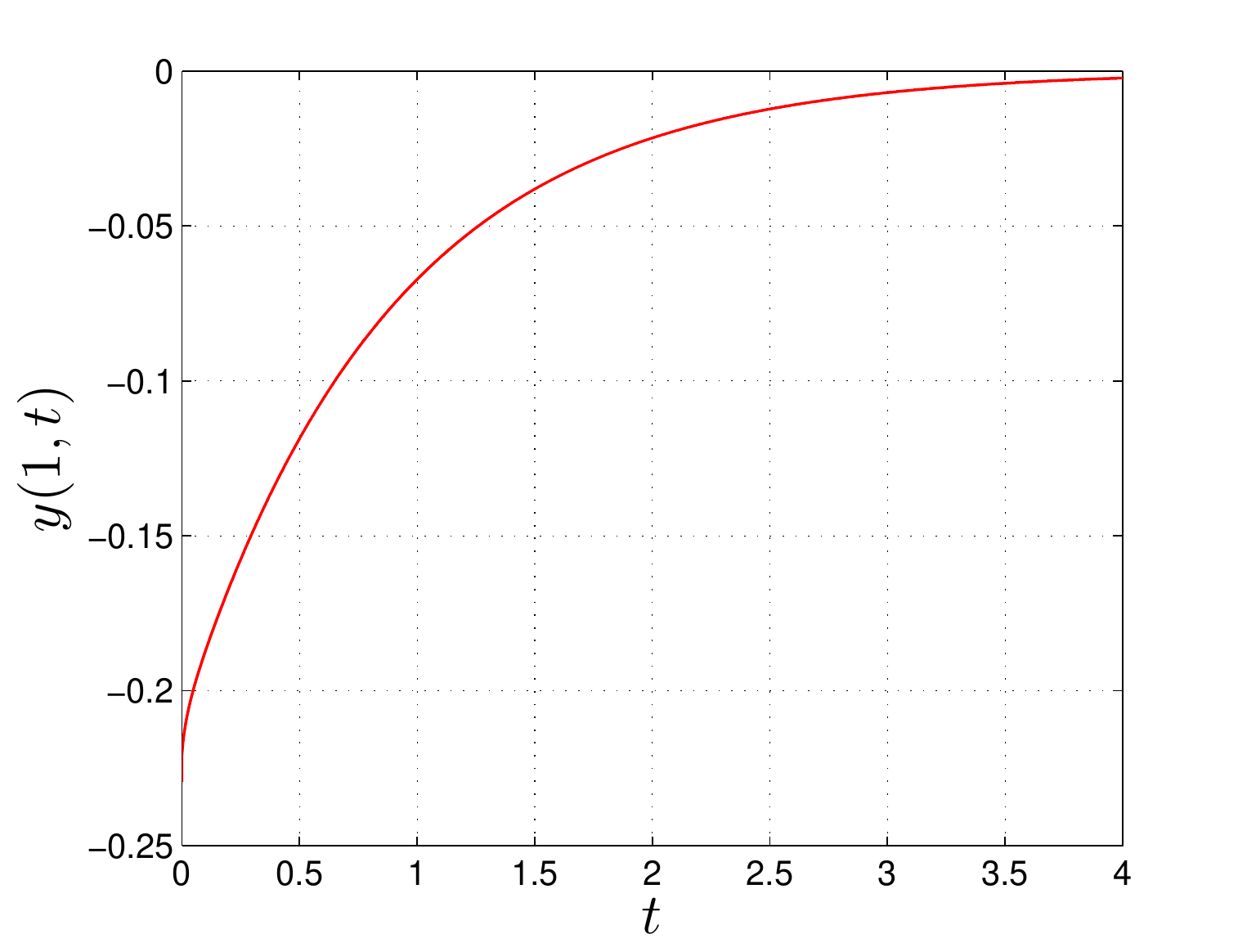}}\label{fig:y_1t-10}
\subfigure[Optimal kernel $k(x)$.] {\includegraphics[width=0.5\textwidth]{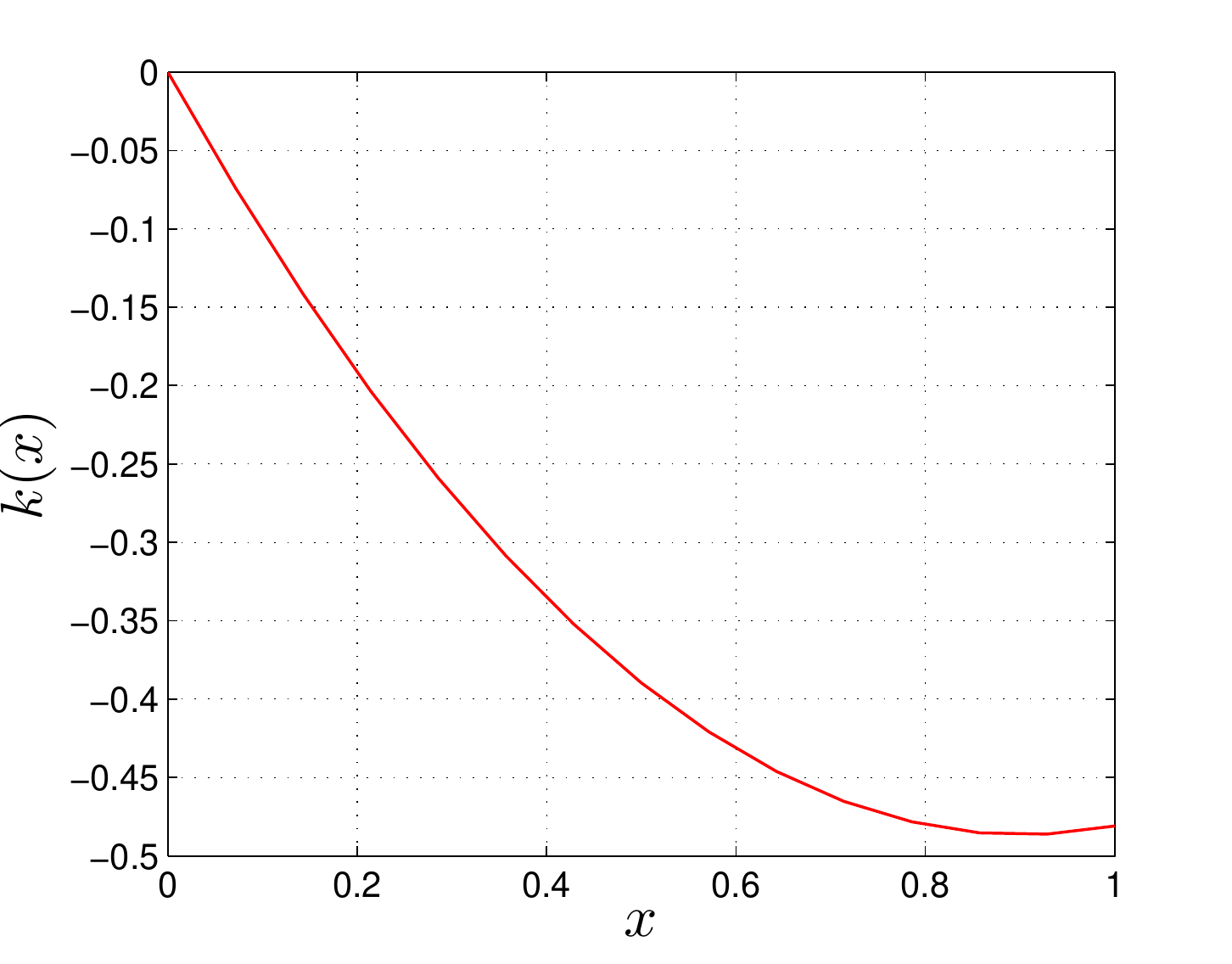}}\label{fig:K_kernels-10}
\subfigure[Left-hand side of (\ref{change-lambda-theda1-theda2}).] {\includegraphics[width=0.50\textwidth]{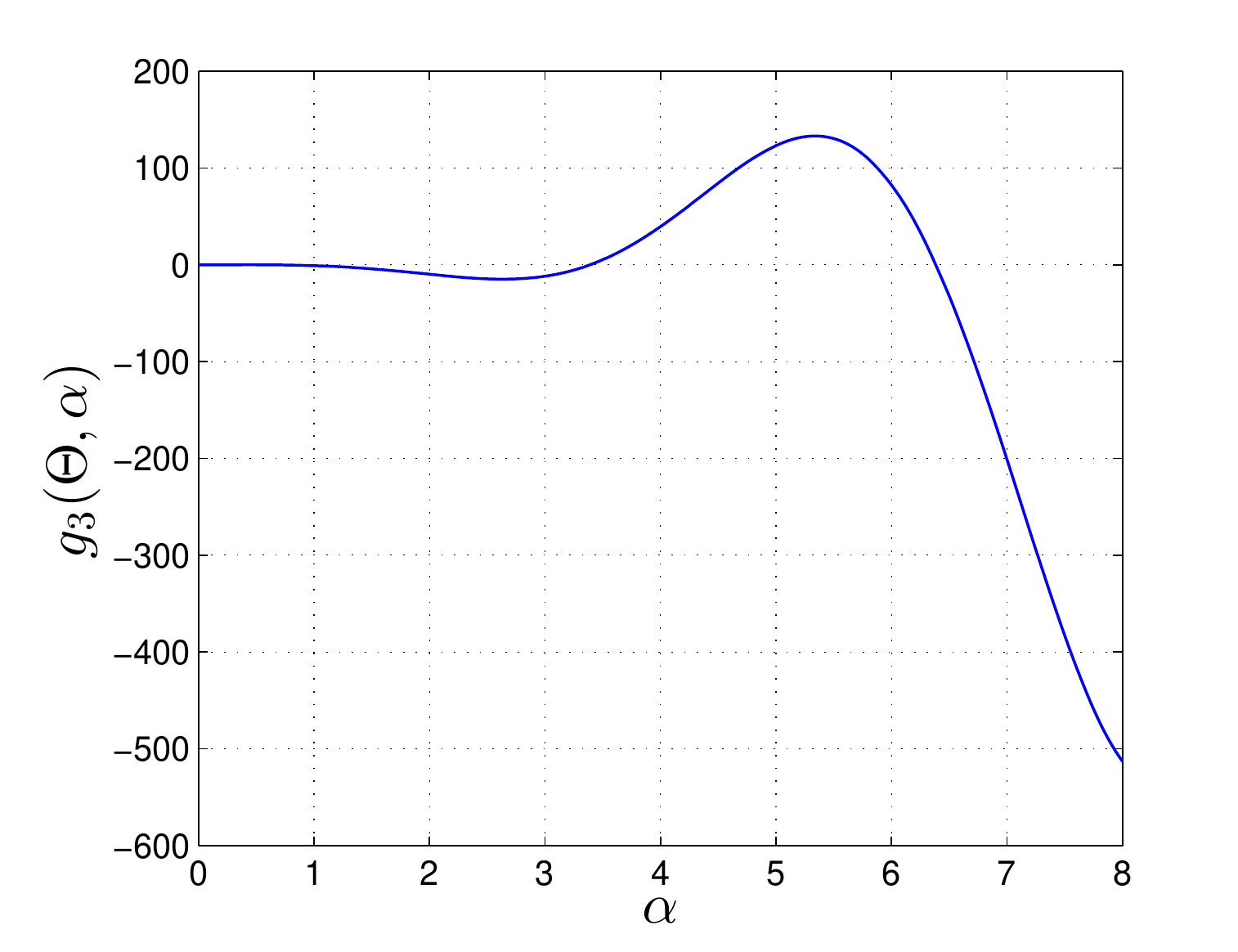}}\label{fig:eigenvalue10}
\caption{ Simulation results for Scenario 1 (optimized parameters $\theta^\ast_{1}=-1.0775$, $\theta^\ast_{2}=0.5966$, $\alpha^\ast=3.3486$). }
\end{figure}

\subsection{Scenario 2}

\begin{table}
\begin{center}
\caption{Solutions of (\ref{change-lambda-theda1-theda2}) and optimal span coefficients $Y_n$ in (\ref{cosfunctionJ}) for Scenario 2.}\label{tab:alpha11}
\begin{tabular}{cccc}
\toprule
$n$ &$\alpha^\ast_n$  &$\alpha^\ast_n/\pi$  &$Y_n$ \\
\toprule
1&  3.6056  &1.1476   & $-0.4185$ \\
2&  6.4595  &2.0562   & 1.4867  \\
3&  9.5520  &3.0404   & 1.4965  \\
4&  12.6561 &4.0285   & $-0.7676$  \\
5&  15.7817 &5.0234   & 0.4462  \\
6&  18.9096 &6.0191   & $-0.3391$  \\
7&  22.0433 &7.0166   & 0.2493   \\
8   &25.1778  &8.0143 & $-0.1992$  \\
9   &28.3147  &9.0128 & 0.1520   \\
10  &31.4520  &10.0114& $-0.1189$  \\
11  &34.5905  &11.0104& 0.0871 \\
12  &37.7291  &12.0095& $-0.0622$ \\
13  &40.8685  &13.0088& 0.0429 \\
14  &44.0081  &14.0086& 1.0061 \\
\toprule
\end{tabular}
\end{center}
\end{table}

For the second scenario, we choose $c=11$ and $y_0(x)=(1+x)\sin(\pi x)$.
~~The ~corresponding uncontrolled~~open-loop ~trajectory ~is ~shown in Figure~\ref{fig:Original U11}.~~Starting from the initial guess $(\theta_{1},\theta_{2}, \alpha)=(-1.0, 1.5, 0)$, our program converges after 26 iterations ~and 11.5767 seconds with an optimal cost value of $g_0=0.5515$. The corresponding optimal parameter values are ~$(\theta^\ast_{1},\theta^\ast_{2}, \alpha^\ast)=(-2.9141, 1.7791, 3.6056)$. We show the spatial-temporal response for the controlled system with optimized feedback parameters $(\theta_1^\ast, \theta_2^\ast )$ in
Figure~\ref{fig:figU11}(a). Again, as with Scenario 1, the controlled plant corresponding to the optimal solution of Problem P$_2$ is stable.  The optimal boundary control and optimal kernel function are shown in Figures ~\ref{fig:figU11}(b) and ~\ref{fig:figU11}(c), respectively. Figure \ref{fig:figU11}(d) shows the left-hand side of (\ref{change-lambda-theda1-theda2}). Note that $\alpha^\ast=3.6056$ is the first positive root, as required by Theorem \ref{theorem0}.
Using MATLAB to minimize (\ref{cosfunctionJ}) for $N=14$, we obtain an optimal cost of $1.249410\times 10^{-12}$, which indicates that the span condition in Theorem \ref{theorem0} holds. The values of $\alpha_n^\ast$ and $Y_n$ in (\ref{cosfunctionJ}) are given in Table \ref{tab:alpha11}.

\begin{figure}
\begin{center}
\includegraphics[width=3.8in]{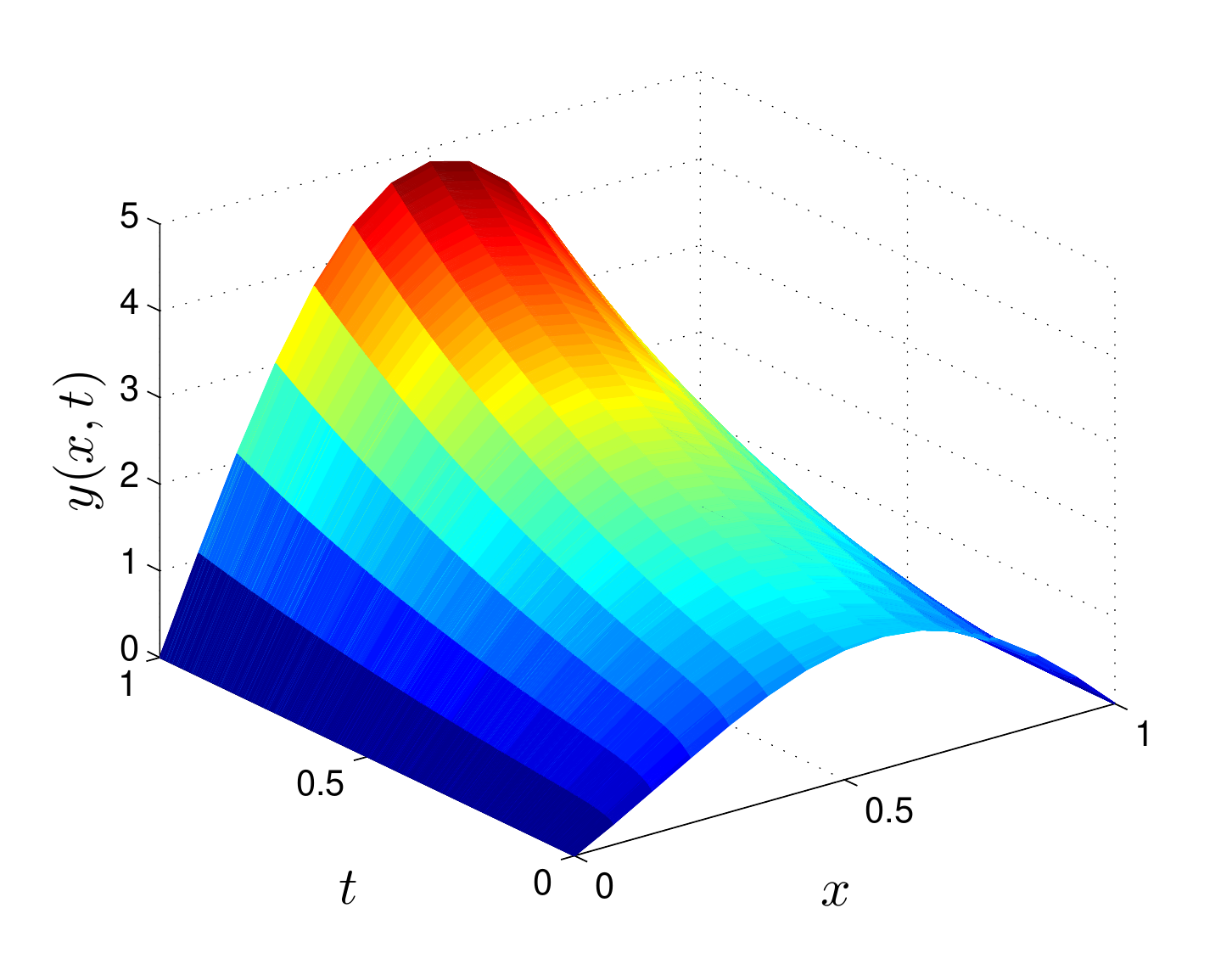}
\caption{Uncontrolled open-loop response for Scenario 2.}\label{fig:Original U11}
\end{center}
\end{figure}

\begin{figure}
\subfigure[Closed-loop response $y(x,t)$.] {\includegraphics[width=0.5\textwidth]{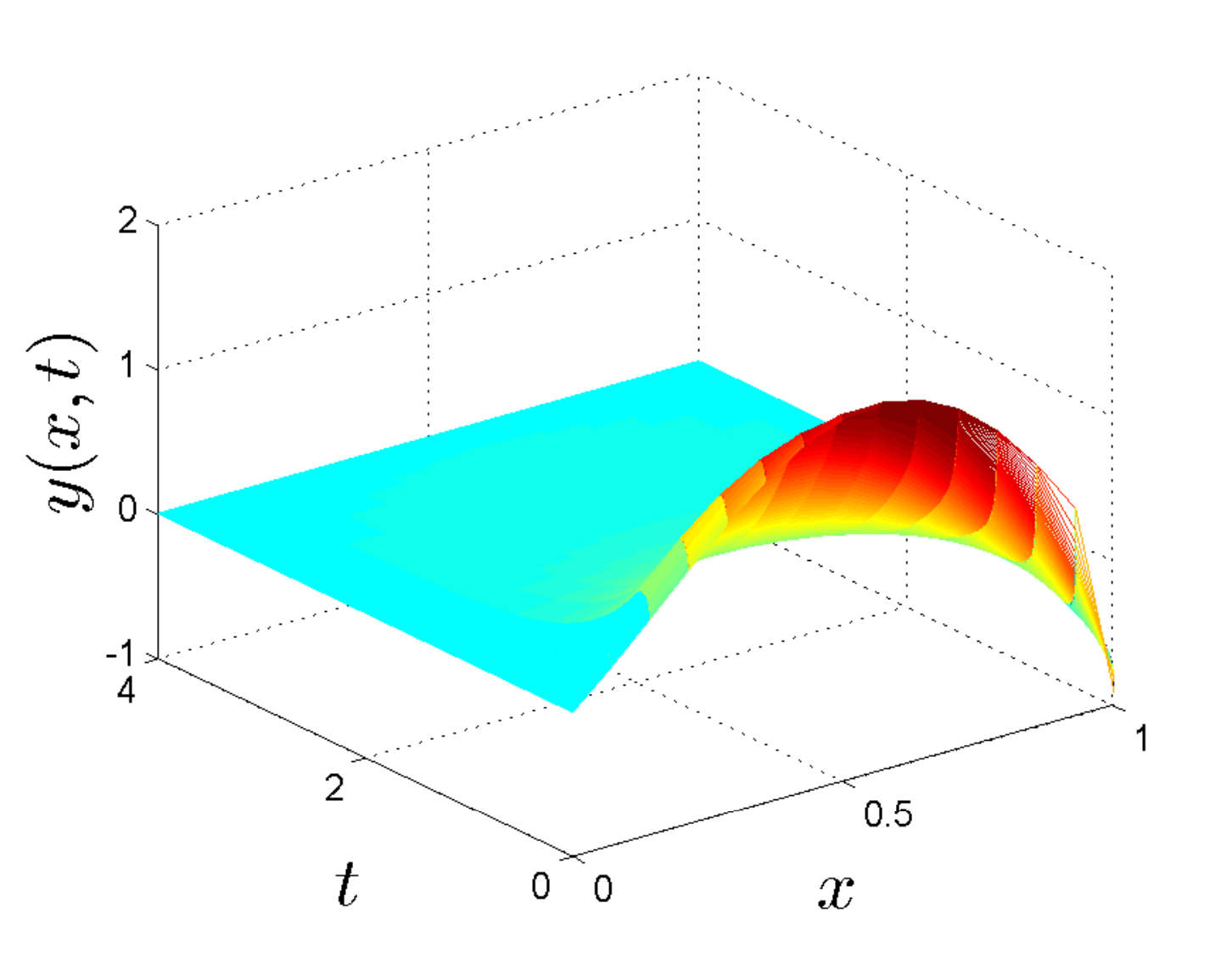}}\label{fig:U1-11}
\subfigure[Optimal boundary control $y(1,t)$.] {\includegraphics[width=0.5\textwidth]{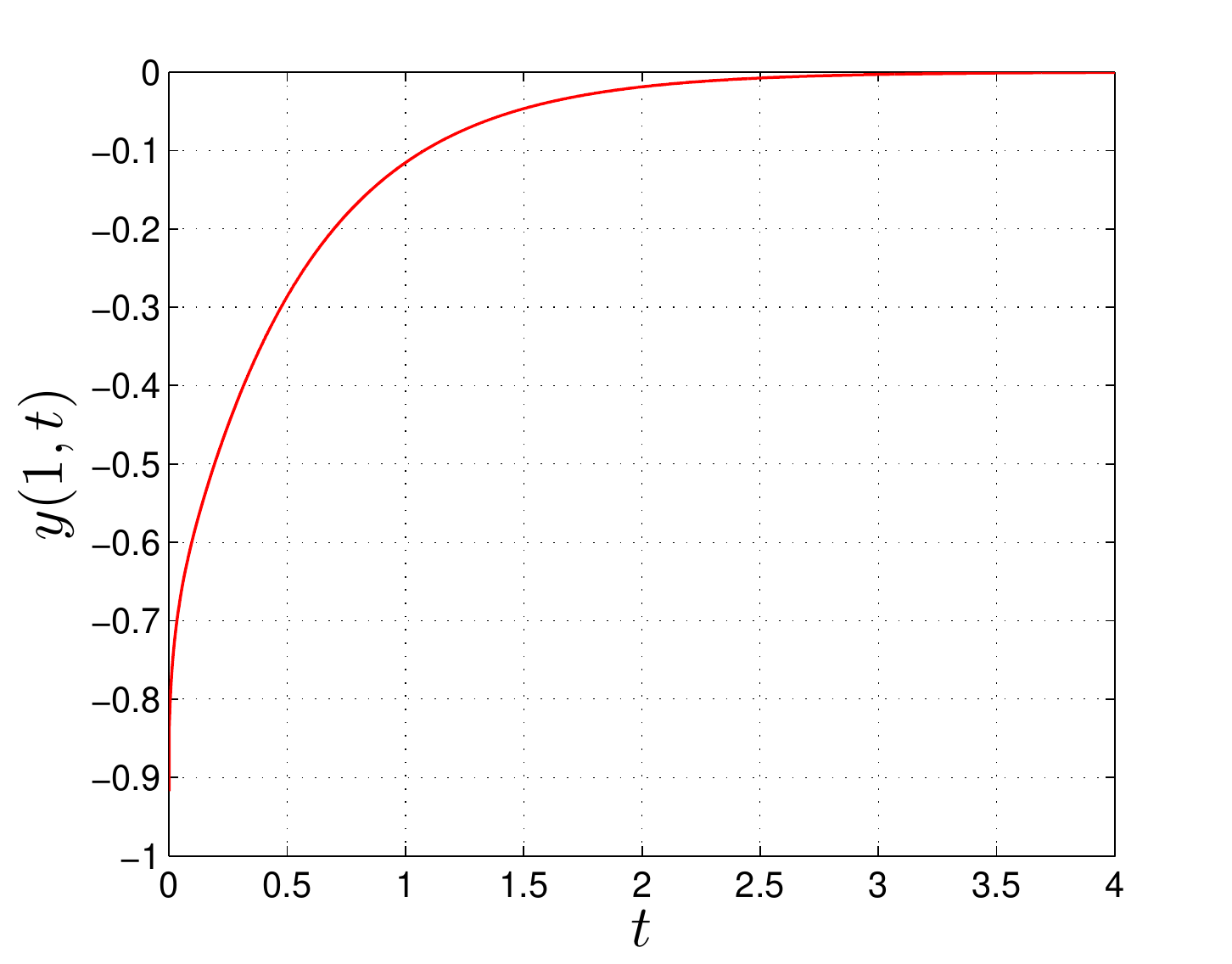}}\label{fig:y_1t-11}
\subfigure[Optimal kernel $k(x)$.] {\includegraphics[width=0.5\textwidth]{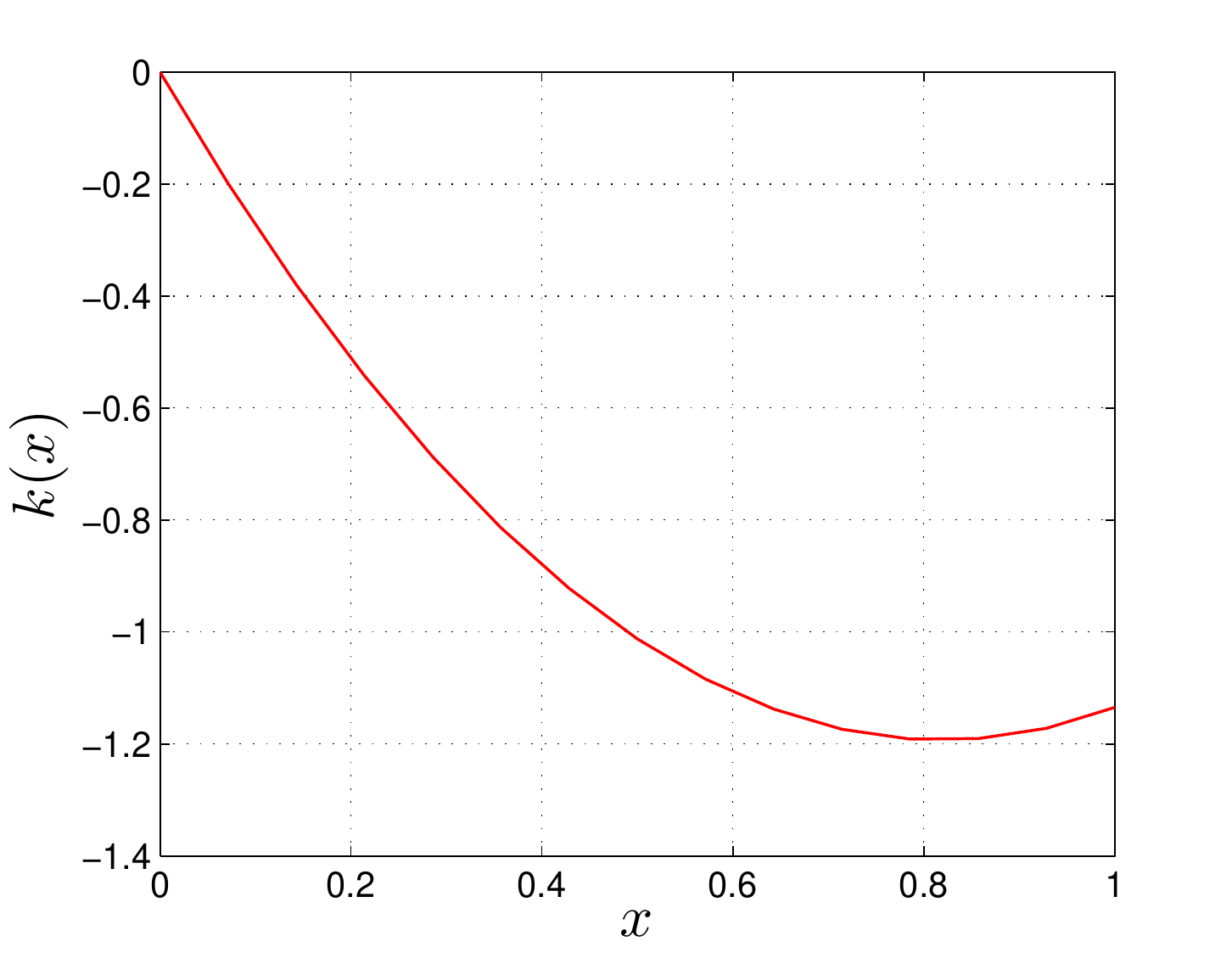}}\label{fig:K_kernels-11}
\subfigure[Left-hand side of (\ref{change-lambda-theda1-theda2}).] {\includegraphics[width=0.49\textwidth]{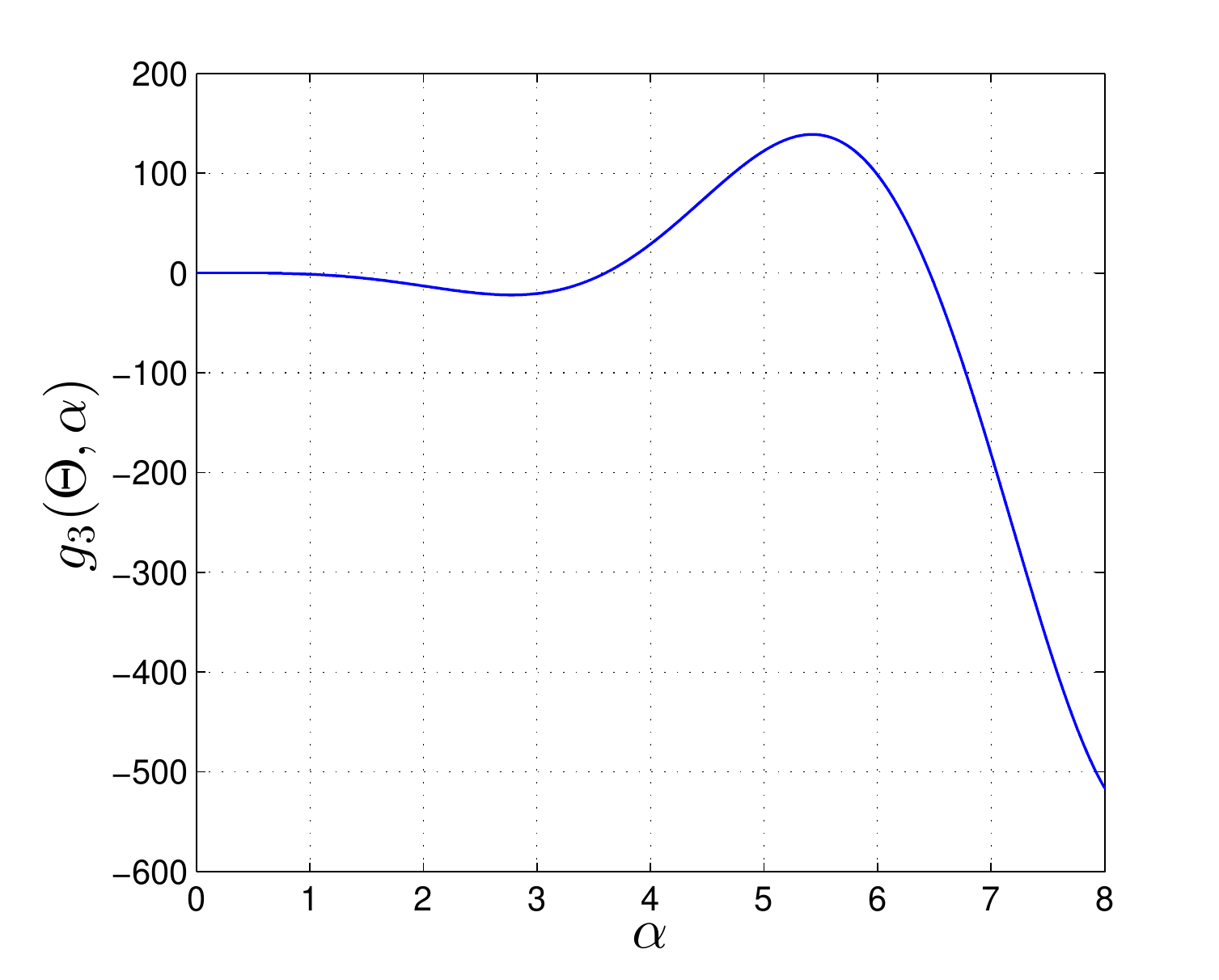}}\label{fig:eigenvalue11}
\caption{ Simulation results for Scenario 2 (optimized parameters $\theta^\ast_{1}=-2.9141$, $\theta^\ast_{2}=1.7791$, $\alpha^\ast=3.6056$). }
\label{fig:figU11}
\end{figure}

\subsection{Scenario 3}
\begin{table}
\begin{center}
\caption{Solutions of (\ref{change-lambda-theda1-theda2}) and corresponding optimal span coefficients $Y_n$ in (\ref{cosfunctionJ}) for Scenario 3.}\label{tab:alpha14}
\begin{tabular}{cccc}
\toprule
$n$ &$\alpha^\ast_n$  &$\alpha^\ast_n/\pi$  &$Y_n$  \\
\toprule
1&  4.1231  &1.3124   & $-0.4383$        \\
2&  6.6959  &2.1314   & 1.7274        \\
3&  9.7345  &3.0986   & 1.2549        \\
4&  12.7804 &4.0683   & $-0.7124$        \\
5&  15.8861 &5.0564   & 0.4225         \\
6&  18.9930 &6.0456   & $-0.3256$         \\
7&  22.1166 &7.0399   & 0.2415            \\
8   &25.2406  &8.0343 & $-0.1951$             \\
9   &28.3713  &9.0308 &  0.1510                 \\
10  &31.5022  &10.0274& $-0.1208$\\
11  &34.6366  &11.0251& 0.0922  \\
12  &37.7711  &12.0229& $-0.0723$\\
13  &40.9075  &13.0212& 0.0662\\
14  &44.0440  &14.0196&  1.0151 \\
\toprule
\end{tabular}
\end{center}
\end{table}
For the final scenario, we choose $c=14$ and $y_0(x)=(2+x)\sin(2.5\pi x)$.
The corresponding uncontrolled open-loop trajectory is shown in Figure~\ref{fig:Original U14}. Starting from the initial guess $(\theta_{1}, \theta_{2}, \alpha)=(-2.0, 1.5, 0)$, our ~program ~terminates ~after ~22 ~iterations ~and ~10.0226 ~seconds with an optimal cost value of $g_0=3.1006$. The corresponding optimal solution is  $(\theta^\ast_{1}, \theta^\ast_{2}, \alpha^\ast)=(-9.1266, 6.4093, 4.1231)$. 
The spatial-temporal response of the controlled plant corresponding to $(\theta^\ast_{1}, \theta^\ast_{2})$ is shown in
Figure~\ref{fig:figU14}(a), which clearly shows that the controlled system (\ref{closed-pde}) with optimized parameters $(\theta^\ast_{1}, \theta^\ast_{2})$ is stable. The optimal boundary control  and optimal kernel function are shown in Figures~\ref{fig:figU14}(b) and~\ref{fig:figU14}(c), respectively. Minimizing (\ref{cosfunctionJ}) for $N=14$ yields an optimal cost of $8.045397\times 10^{-15}$. We report the corresponding values of $\alpha_n^\ast$ and $Y_n$ in Table \ref{tab:alpha14}. Finally, Figure \ref{fig:figU14}(d) shows the left-hand side of equation (\ref{change-lambda-theda1-theda2}) corresponding to the optimized parameters.



\begin{figure}
\begin{center}
\includegraphics[width=3.8in]{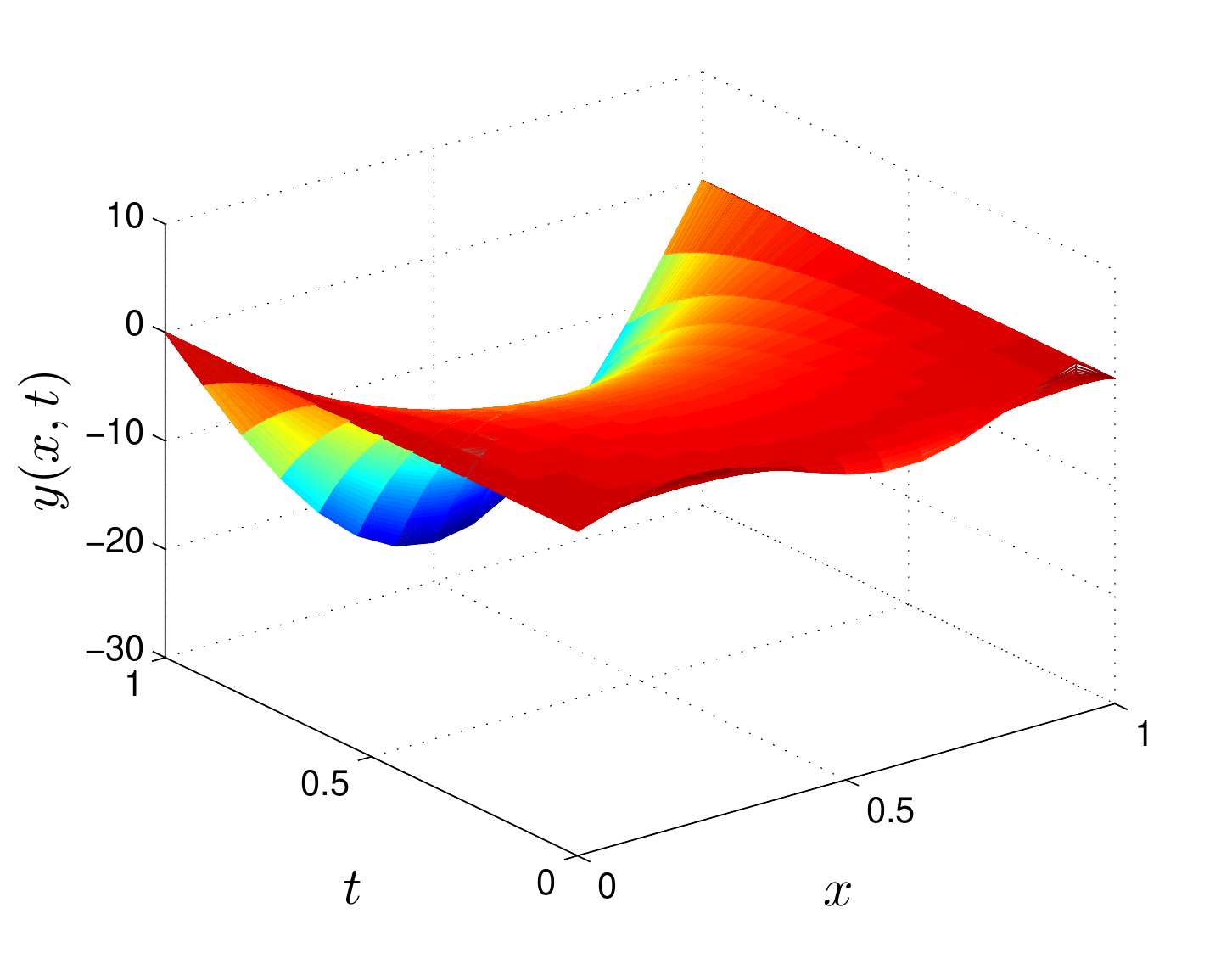}
\caption{Uncontrolled open-loop response for Scenario 3.}\label{fig:Original U14}
\end{center}
\end{figure}

\begin{figure}
\subfigure[Closed-loop response $y(x,t)$.] {\includegraphics[width=0.5\textwidth]{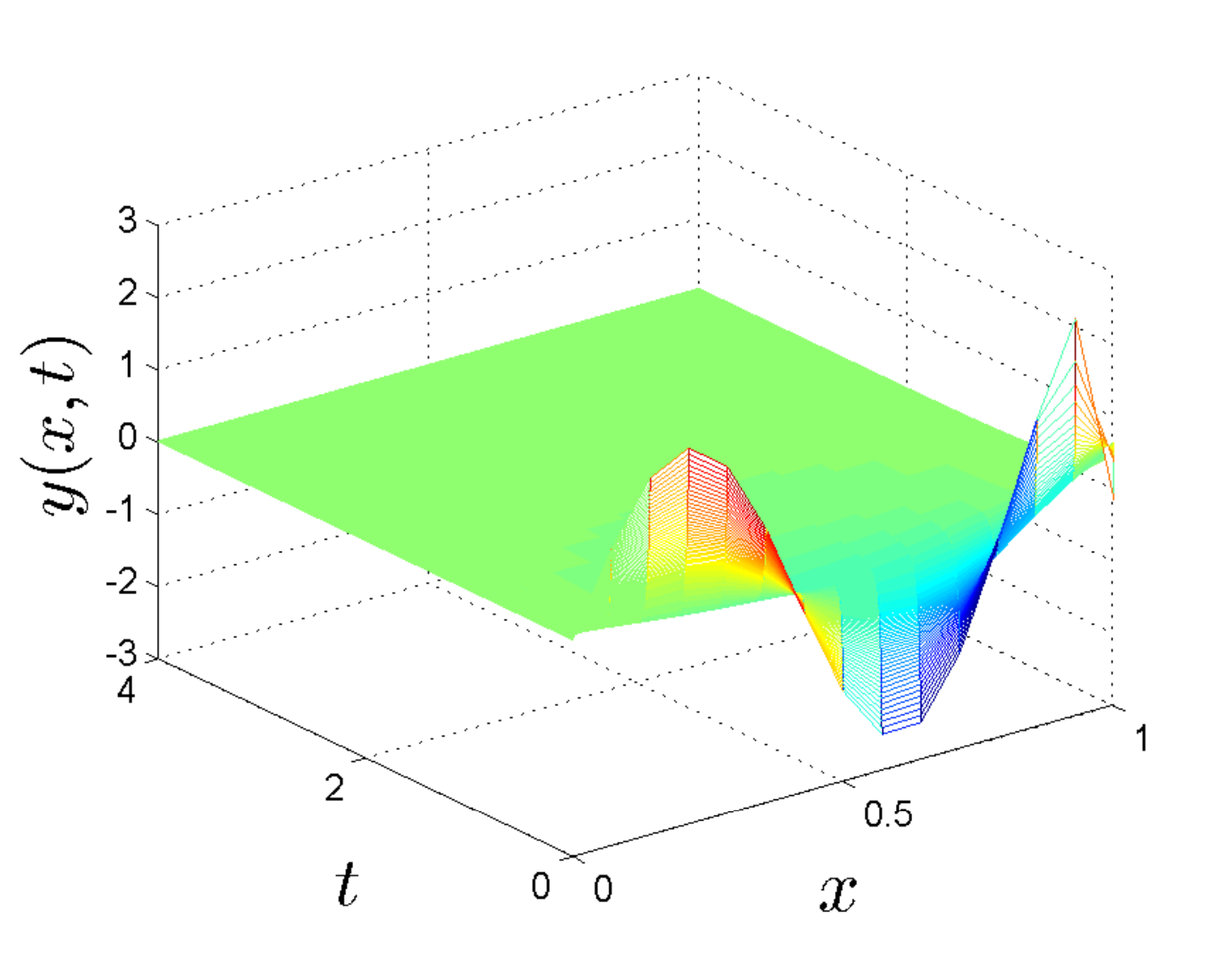}}\label{fig:U1-14}
\subfigure[Optimal boundary control $y(1,t)$.] {\includegraphics[width=0.5\textwidth]{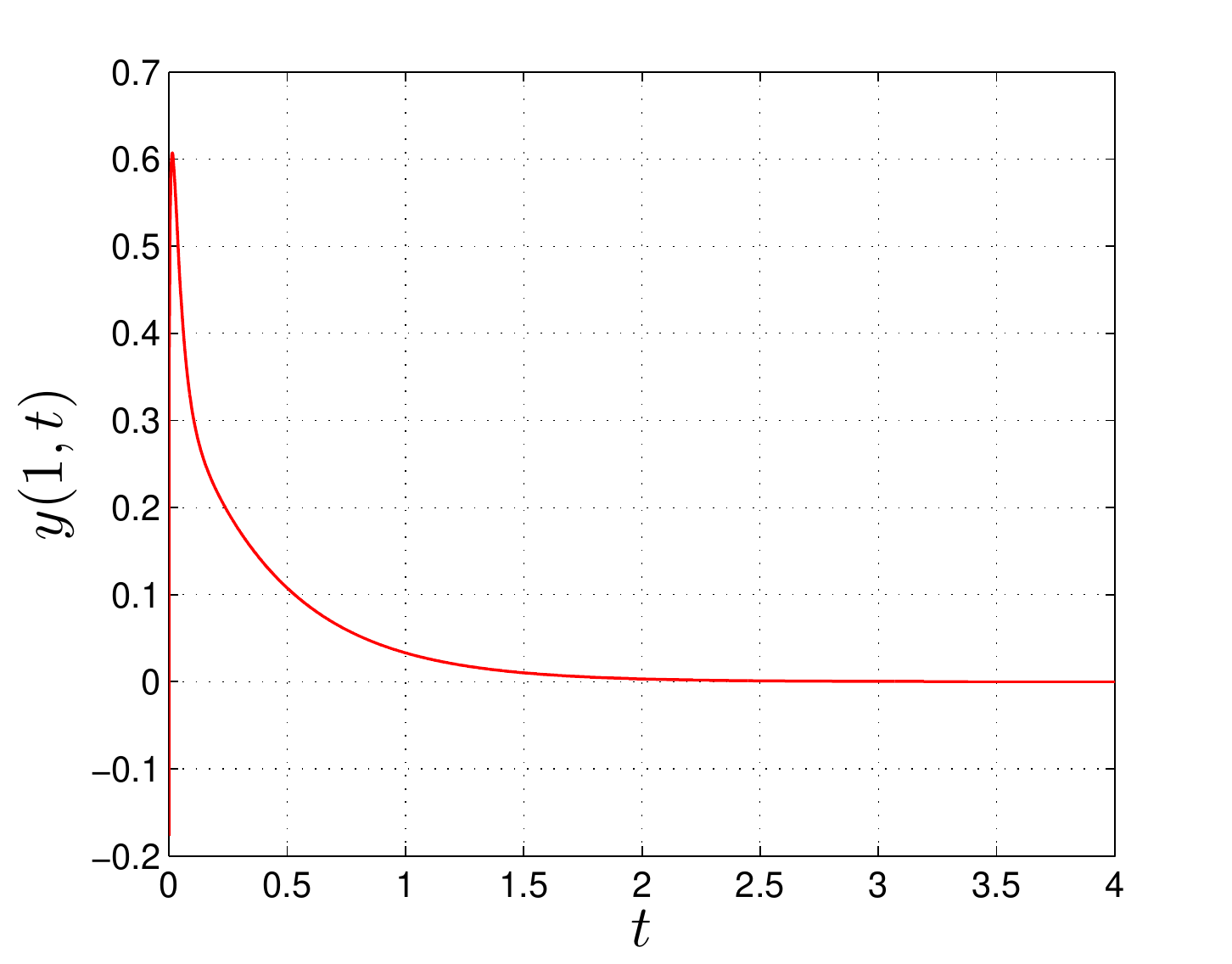}}\label{fig:y_1t-14}
\subfigure[Optimal kernel $k(x)$.] {\includegraphics[width=0.5\textwidth]{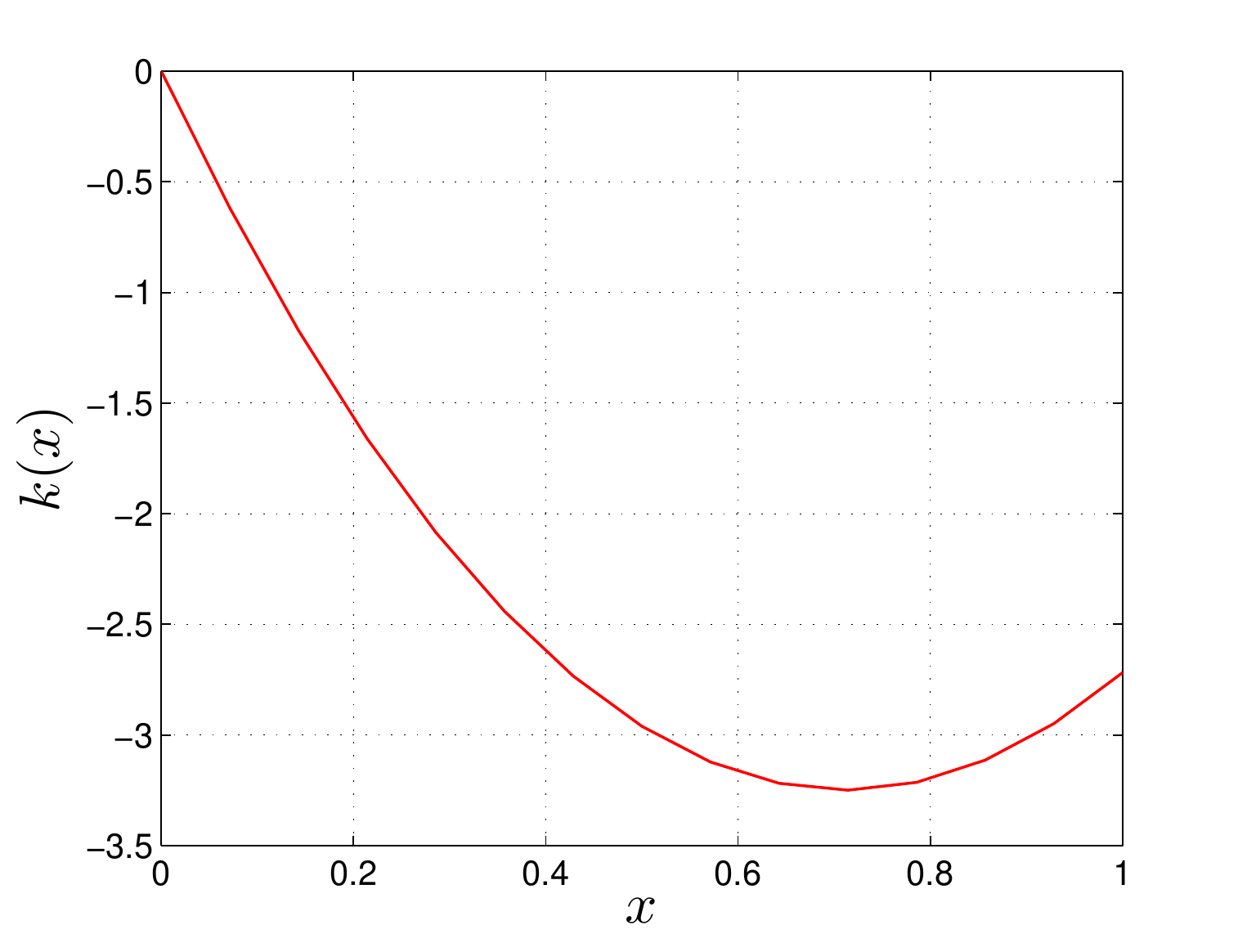}}\label{fig:K_kernels-14}
\subfigure[Left-hand side of (\ref{change-lambda-theda1-theda2}).] {\includegraphics[width=0.48\textwidth]{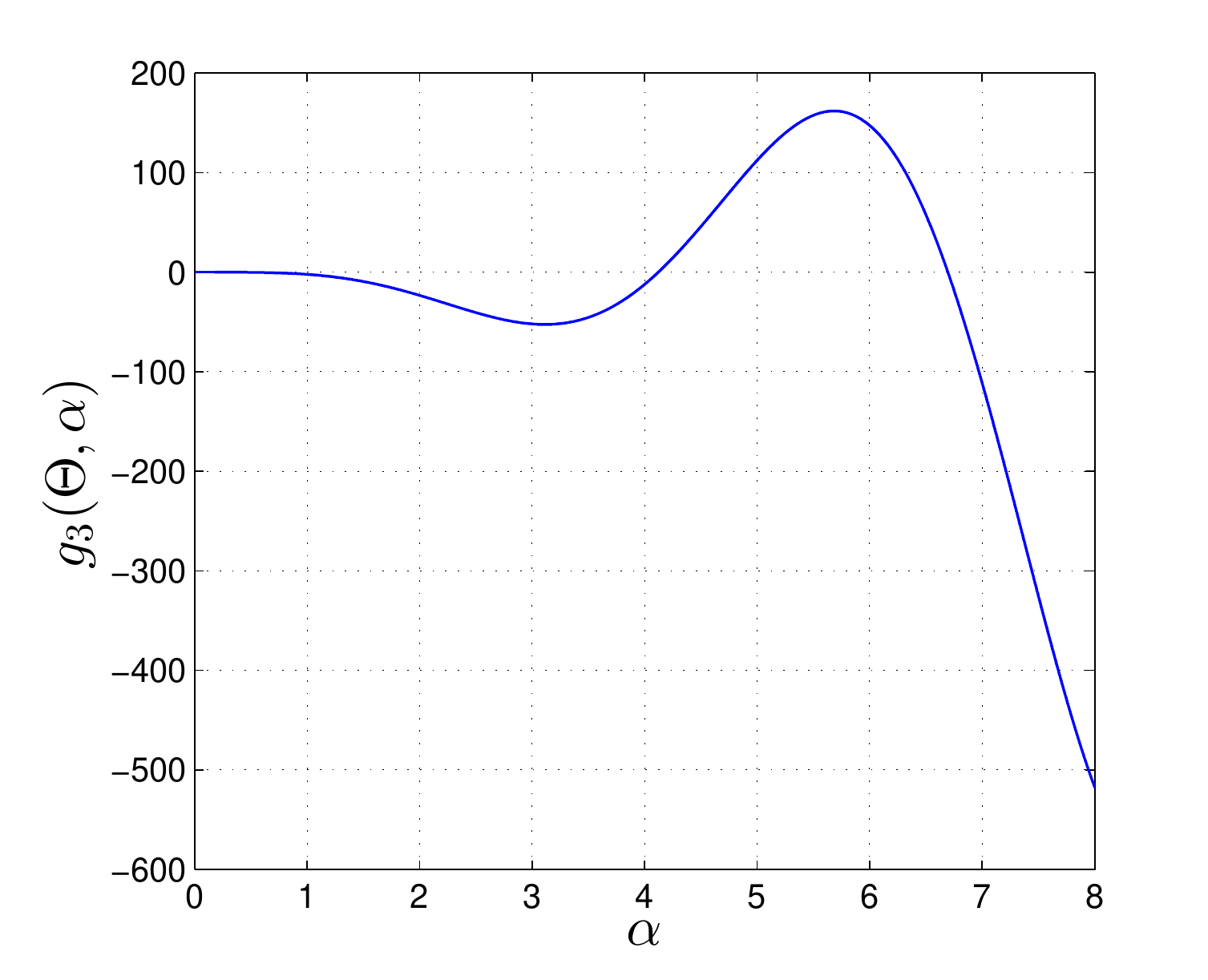}}\label{fig:eigenvalue14}
\caption{ Simulation results for Scenario 3 (optimized parameters $\theta^\ast_{1}=-9.1266$, $\theta^\ast_{2}=6.4093$, $\alpha^\ast=4.1231$). }
\label{fig:figU14}
\end{figure}
\section{Conclusion}\label{sec:conclusion}
In this paper, we have introduced a new gradient-based optimization approach for boundary stabilization of parabolic PDE systems. As with the well-known LQ control and backstepping synthesis approaches, our new approach involves expressing the boundary controller as an integral state feedback in which a kernel function needs to be designed judiciously. However, unlike the LQ control and backstepping approaches, we do not determine the feedback kernel by solving  Riccati-type or Klein-Gorden-type PDEs; instead, we approximate the feedback kernel by a quadratic function and then optimize the quadratic's coefficients using dynamic optimization techniques. This approach requires solving a so-called ``costate PDE'', which is much easier to solve numerically than the Riccati and Klein-Gorden PDEs. Indeed, as shown in Section \ref{sec:NumericalAlgorithms}, the costate PDE can be solved easily using the finite-difference method. Based on the work in this paper, we have identified several unresolved research questions described as follows: ~(i) Is it possible to prove, ~or at least weaken, the ~linear ~span ~condition ~in ~Theorem \ref{theorem0}? ~(ii) Can the proposed kernel optimization approach be applied to other classes of PDE plant models? ~(iii) Is it possible to develop methods for minimizing cost functional (\ref{optcostfun}) over an infinite time horizon? These issues will be explored in future work.
\section*{Acknowledgements}
This paper is dedicated to Professor Kok Lay Teo on the occasion of his 70th birthday. Professor Teo has been valued  colleague and mentor to each of the authors of this paper. The gradient computation procedure in Section \ref{sec:VariationalAnalysis} was inspired by Professor Teo's seminal work in \cite{Teo1991}; see also the recent survey papers \cite{lin2013optimal,linsurvey2013}.

\appendix

\section{Proof of Lemma \ref{lemma1}}

We prove the lemma in three steps.

\subsection{Preliminaries}
Let
\begin{equation*}\label{changeQx}
\begin{aligned}
Q(\alpha)&=\sqrt{(\theta_1\alpha^2+\theta_2\alpha^2-2\theta_2)^2+(\alpha^3-\theta_1\alpha-2\theta_2\alpha)^2}\\
&=\sqrt{\alpha^6+(\theta_1^2+2\theta_1\theta_2+\theta_2^2-2\theta_1-4\theta_2)\alpha^4+\theta_1^2\alpha^2+4\theta_2^2}.
\end{aligned}
\end{equation*}
Furthermore, let $\varphi(\alpha)\in (-\pi,\pi]$ be the unique angle satisfying
\begin{equation*}\label{cosfunction}
\cos(\varphi(\alpha))=\frac{\alpha^3-\theta_1\alpha-2\theta_2\alpha}{Q(\alpha)},
\end{equation*}
and
\begin{equation*}\label{sinfunction}
\sin(\varphi(\alpha))=\frac{\theta_1\alpha^2+\theta_2\alpha^2-2\theta_2}{Q(\alpha)}.
\end{equation*}
Using the definitions of $Q(\alpha)$ and $\varphi(\alpha)$, equation (\ref{change-lambda-theda1-theda2}) can be rewritten as follows:
\begin{equation*}
Q(\alpha)\sin(\varphi(\alpha))\cos(\alpha)+Q(\alpha)\cos(\varphi(\alpha))\sin(\alpha)=-2\theta_2.
\end{equation*}
Thus, using the angle sum trigonometric identity, we obtain
\begin{equation}\label{change-lambda-theda1-theda3}
Q(\alpha)\sin(\alpha+\varphi(\alpha))=-2\theta_2.
\end{equation}
Now, under condition (\ref{constraint00}), $Q(\alpha)\rightarrow \infty$ as $\alpha\rightarrow \infty$.
Furthermore,
\begin{equation*}
 \lim_{\alpha\rightarrow +\infty}\cos(\varphi(\alpha))=1, \quad \lim_{\alpha\rightarrow +\infty}\sin(\varphi(\alpha))=0.
\end{equation*}
Hence, $\varphi(\alpha)\rightarrow 0$ as $\alpha\rightarrow \infty$.
\subsection{Angle $\varphi(\alpha)$ is  Continuous at all Sufficiently Large $\alpha$}
Since $\varphi(\alpha)\rightarrow 0$ as $\alpha\rightarrow \infty$, there exists a constant $\bar \alpha$ such that $-\frac{1}{4}\pi<\varphi(\alpha)<\frac{1}{4}\pi$ for all $\alpha>\bar\alpha$. Consider an arbitrary point $\alpha'>\bar\alpha$. We will show that $\varphi(\cdot)$ is continuous at $\alpha'$.

Let $\delta>0$. In view of the definition of $\sin(\varphi(\alpha))$, there exists an $\varepsilon>0$ such that
\begin{equation}\label{A1}
|\alpha-\alpha'|<\varepsilon \quad\implies\quad -\frac{1}{\sqrt{2}}\delta<\sin(\varphi(\alpha))-\sin(\varphi(\alpha'))<\frac{1}{\sqrt{2}}\delta.
\end{equation}
Now, using Taylor's Theorem,
\begin{equation}\label{A2}
\sin(\varphi(\alpha))-\sin(\varphi(\alpha'))=\cos(\zeta)(\varphi(\alpha)-\varphi(\alpha')),
\end{equation}
where $\zeta$ belongs to the interval bounded by $\varphi(\alpha)$ and $\varphi(\alpha')$. Suppose $\alpha$ satisfies $|\alpha-\alpha'|<\min(\varepsilon,\alpha'-\bar\alpha)$.
Then
\begin{equation*}
-\frac{1}{4}\pi<\varphi(\alpha)<\frac{1}{4}\pi, \quad -\frac{1}{4}\pi<\varphi(\alpha')<\frac{1}{4}\pi.
\end{equation*}
Hence,
\begin{equation*}
-\frac{1}{4}\pi<\zeta<\frac{1}{4}\pi
\end{equation*}
and
\begin{equation}\label{A3}
\cos\zeta>\frac{1}{\sqrt{2}}.
\end{equation}
Combining (\ref{A1})-(\ref{A3}) yields
\begin{equation*}
\frac{1}{\sqrt{2}}\delta>|\sin(\varphi(\alpha))-\sin(\varphi(\alpha'))|=|\cos(\zeta)|\cdot|\varphi(\alpha)-\varphi(\alpha')|
\geq\frac{1}{\sqrt{2}}|\varphi(\alpha)-\varphi(\alpha')|.
\end{equation*}
Hence, we  have established the following implication:
\begin{equation*}
|\alpha-\alpha'|<\min(\varepsilon,\alpha'-\bar\alpha)\quad \implies \quad  |\varphi(\alpha)-\varphi(\alpha')|<\delta.
\end{equation*}
This shows that $\varphi(\cdot)$ is continuous at $\alpha'$, as required.

\subsection{Roots of Equation (A.1)}
Let $\epsilon\in(0,\frac{1}{2}\pi)$ and define
\begin{equation*}
  a_k=k\pi-\epsilon, \quad b_k=k\pi+\epsilon.
\end{equation*}
Clearly, for each integer $k\geq 0$, $a_k<b_k<a_{k+1}$ and
\begin{equation*}
\sin(a_k)=
\begin{cases}
\sin(\epsilon),&\text{if $k$ is odd},\\
-\sin(\epsilon),&\text{if $k$ is even},
\end{cases}
\end{equation*}
\begin{equation*}
\sin(b_k)=
\begin{cases}
-\sin(\epsilon),&\text{if $k$ is odd},\\
\sin(\epsilon),&\text{if $k$ is even}.
\end{cases}
\end{equation*}
Using Taylor's Theorem, we have
\begin{equation*}
  \sin(\alpha+\varphi(\alpha))=\sin(\alpha)+\cos(\zeta)\varphi(\alpha),
\end{equation*}
where $\zeta=\zeta(\alpha)$ belongs to the interval bounded by $\alpha$ and $\alpha+\varphi(\alpha)$.
Thus,
\begin{equation}\label{equ:tylorequation}
  |\sin(\alpha+\varphi(\alpha))-\sin(\alpha)|=|\cos(\zeta)\varphi(\alpha)|\leq |\varphi(\alpha)|.
\end{equation}
Since $\varphi(\alpha)\rightarrow 0$ as $\alpha\rightarrow \infty$, there exists
an integer $k_1\geq 1$ such that for all $k\geq k_1$,
\begin{equation*}
\begin{aligned}
  |\varphi(a_k)|< \frac{1}{2}\sin(\epsilon), \quad
  |\varphi(b_k)|< \frac{1}{2}\sin(\epsilon).
\end{aligned}
\end{equation*}
Hence, substituting $\alpha=a_k$ and $\alpha=b_k$ into (\ref{equ:tylorequation}) gives, for $k\geq k_1$,
\begin{equation*}
\sin(a_k+\varphi(a_k))
\begin{cases}
> \frac{1}{2}\sin(\epsilon),&\text{if $k$ is odd},\\
< -\frac{1}{2}\sin(\epsilon),&\text{if $k$ is even},
\end{cases}
\end{equation*}
\begin{equation*}
\sin(b_k+\varphi(b_k))
\begin{cases}
< -\frac{1}{2}\sin(\epsilon),&\text{if $k$ is odd},\\
> \frac{1}{2}\sin(\epsilon),&\text{if $k$ is even}.
\end{cases}
\end{equation*}
Since $Q(\alpha)\rightarrow \infty$ as $\alpha\rightarrow \infty$, there exists an integer $k_2\geq 1$ such that for all $k\geq k_2$,
\begin{equation*}\label{changeQx}
-\frac{1}{2} Q(a_k)\sin(\epsilon)\leq-2\theta_2\leq\frac{1}{2} Q(a_k)\sin(\epsilon),
\end{equation*}
\begin{equation*}\label{changeQx}
-\frac{1}{2}Q(b_k)\sin(\epsilon) \leq-2\theta_2\leq  \frac{1}{2}Q(b_k)\sin(\epsilon).
\end{equation*}
Thus, for all $k\geq\max\{k_1,k_2\}$,
\begin{equation*}
Q(a_k)\sin(a_k+\varphi(a_k))
\begin{cases}
> \frac{1}{2}Q(a_k)\sin(\epsilon)\geq-2\theta_2,&\text{if $k$ is odd},\\
< -\frac{1}{2}Q(a_k)\sin(\epsilon)\leq-2\theta_2,&\text{if $k$ is even},
\end{cases}
\end{equation*}
\begin{equation*}
Q(b_k)\sin(b_k+\varphi(b_k))
\begin{cases}
< -\frac{1}{2}Q(b_k)\sin(\epsilon)\leq-2\theta_2,&\text{if $k$ is odd},\\
> \frac{1}{2}Q(b_k)\sin(\epsilon)\geq-2\theta_2,&\text{if $k$ is even}.
\end{cases}
\end{equation*}
Since $\varphi$ is continuous when $\alpha$ is large, this implies that, for all sufficiently large $k$,  there exists a solution of (\ref{change-lambda-theda1-theda3}) within the interval $[a_{k},b_k]$. The result follows immediately.


\end{document}